\documentclass[11pt]{amsart}

\usepackage[draft]{changes}

\definechangesauthor[color=red]{zl}
\definechangesauthor[color=blue]{hh}
\definechangesauthor[color=orange]{hx}
\newcommand{\stkout}[1]{\ifmmode\text{\sout{\ensuremath{#1}}}\else\sout{#1}\fi}
\setdeletedmarkup{\stkout{#1}}

\usepackage{amssymb} 

\usepackage{graphics}
\usepackage{graphicx}

\usepackage{amsmath} 
\usepackage{amsthm}
\usepackage{amsfonts}
\usepackage{xcolor}
\usepackage[english]{babel}
\usepackage[margin=1.0in]{geometry}
\usepackage[colorlinks=true,
        linkcolor=blue]{hyperref}
\usepackage{bbm}
\usepackage{verbatim}
\usepackage{mathrsfs}

\numberwithin{equation}{section}

\newtheorem{prop}{Theorem}
\newtheorem{lemma}[prop]{Lemma}

\newtheorem{thm}[prop]{Theorem}
\newtheorem{cor}[prop]{Corollary}
\newtheorem{conj}[prop]{Conjecture}

\numberwithin{prop}{section}

\newtheorem{defn}[prop]{Definition}
\theoremstyle{definition}

\newtheorem{rmk}[prop]{Remark}

\definecolor{c1}{rgb}{0.2,0.4,0.5}
\definecolor{c2}{rgb}{0.1,0.3,0.5}
\definecolor{c3}{rgb}{0.2,0.7,0.5}
\usepackage{tikz}

\def \k {K\"ahler }
\newcommand{\oo}[1]{\overline{#1}}

\newcommand{\del}{\partial}
\newcommand{\bdel}{\bar{\partial}}

\newcommand{\gw}{\omega}

\newcommand{\ten}{\otimes}

\newcommand{\dbar}{\oo\partial_{k}}
\newcommand{\dbarb}{\oo\partial_{b}}

\newcommand{\eps}{\varepsilon}
\newcommand{\dstar}{\oo{\partial}^*_k}

\newcommand{\Dbar}{\oo{D}_k}
\newcommand{\Dstar}{\oo{D}_k^*}

\DeclareMathOperator{\supp}{supp}

\DeclareMathOperator{\Ric}{Ric}

\DeclareMathOperator{\Ima}{Im}

\begin{document}

\title[Quantitative upper bounds for Bergman kernels]{Quantitative upper bounds for Bergman kernels associated to smooth K\"ahler potentials}

\begin{abstract} We prove upper bounds for the Bergman kernels associated to tensor powers of a smooth positive line bundle in terms of the rate of growth of the Taylor coefficients of the \k potential. As applications, we obtain improved off-diagonal rate of decay for the classes of analytic, quasi-analytic, and more generally Gevrey potentials.  
\end{abstract}


\author [Hezari]{Hamid Hezari}
\address{Department of Mathematics, UC Irvine, Irvine, CA 92617, USA} \email{\href{mailto:hezari@uci.edu}{hezari@uci.edu}}

\author [Xu]{Hang Xu}
\address{Department of Mathematics, Johns Hopkins University, Baltimore, MA 21218, USA}
\email{\href{mailto:hxu@math.jhu.edu}{hxu@math.jhu.edu}}

\maketitle

\section{Introduction}

Let $(L,h) \to X$ be a positive Hermitian holomorphic line bundle over a compact complex manifold of dimension $n$. The metric $h$ induces the \k form $\gw= -\tfrac{\sqrt{-1}}{2} \del \bdel \log(h)$ on $X$.  For $k$ in $\mathbb N$, let $H^0(X,L^k)$ denote the space of holomorphic sections of $L^k$. The {Bergman projection} is the orthogonal projection $\Pi_k: {L}^{2}(X,L^k) \to H^0(X,L^k)$ with respect to the natural inner product induced by the metric $h^k$ and the volume form $dV=\frac{ \gw^n }{n!}$. The \emph{Bergman kernel} $B_k$, a section of $L^k\ten \oo{L^k}$, is the distribution kernel of $\Pi_k$.
Given $p \in X$,  let $(U, e_L)$ be a local trivialization of $L$ near $p$.   We write $| e_L |^2_{h}=e^{-\phi}$ and call $\phi$ a local \k potential. In the frame $ e_L^{k} \ten {\bar{e}_L^{k}}$, the Bergman kernel $B_k(z,w)$ is understood as a function $K_k(z, w)$ on $U \times U$. We note that on the diagonal $z=w$, the function $K_k(z,z)e^{-k\phi(z)}$ is independent of the choice of the local frame, hence it is a globally defined function on $X$  called the \emph{Bergman function}, which is also equal to $| B_k(z, z)  |_{h^k}$. In addition, by our notations, $| B_k(z, w)  |_{h^k}= |K_k(z,w) |e^{-k\phi(z) /2 -k\phi(w) /2}$. To motivate our problem, we start with the following refinement of a result of M. Christ \cite{Ch2} which improves the Agmon estimates.

\begin{thm}\label{christ} 
Assume $h \in C^\infty$. Let $d(z,w)$ be the distance function on $X$ associated with the \k metric $\omega$. Then there exists an increasing function $f(k) \to \infty$ as $k \to \infty$ such that for all $z$ and $w$ in $X$, we have
\begin{equation}\label{eq christ}
|B_k(z,w)|_{h^k} \leq 
\begin{cases}
C k^{n} e^{-c \, k d(z,w)^2}, \quad & \mbox{ when } d(z,w)\leq f(k)\sqrt{\frac{\log k}{k}},\\
C k^{n} e^{-c \, f(k)\sqrt{k\log k} \,\, d(z, w)}, \quad & \mbox{ when } d(z,w) \geq f(k)\sqrt{\frac{\log k}{k}}.
\end{cases}
\end{equation} 
where $c$ and $C$ are positive constants that depend only on $(L, h)$ and $X$. 
\end{thm}
Although in \cite{Ch2}, this result is only stated for $d(z,w) > \delta$, for $\delta>0$ independent of $k$, it is implicit in their proof.  The goal of this paper is to prove a quantitative version of the above theorem that relates the growth rate of $f(k)$ to the growth rate of the derivatives of $h$,  and use it to obtain improvements in the cases where $h$ is analytic, quasi-analytic, or more generally a Gevrey function. 

To state our main result we need some notations. Let $M$ be a positive $C^2$ function on $\mathbb{R}_{>0}$. We say $\partial^2\log h$ belongs to the class $\mathcal C_M$ if for every local \k potential $\phi$ on an open set $U$, there exists $A>0$ such that for any multi-index $\alpha \in \mathbb{Z}_{\geq 0}^{2n}$ with $| \alpha| \neq 0$, and any $z\in U$, 
		\begin{equation*}
			\left|\frac{D^{\alpha} (\partial^2\phi )}{\alpha!}(z)\right|\leq A^{|\alpha|} {M( |\alpha|)}.
		\end{equation*}
Here $D^\alpha = \partial_{z_1}^{\alpha_1} \dots \partial_{z_n}^{\alpha_n}\partial_{\bar{z}_1}^{\alpha_{n+1}} \dots \partial_{\bar{z}_n}^{\alpha_{2n}}$, and $\partial^2 \phi$ means any second order partial derivative of $\phi$. We shall call $M$ a majorant of the Taylor coefficients of $\partial^2 \phi$. 

\begin{thm}\label{main} Suppose $\partial^2 \log h \in \mathcal C_M$ and $\log \left({M}\right)$ is strictly convex on $\mathbb{R}_{>0}$. For each $N \in {\mathbb R}_{>0}$, let $J(N)=M(N)^{1/N}$ and assume $J(N)$ is unbounded. Then there exists $k_0$ such that for each $k \geq k_0$, the equation 
\begin{equation}\label{solution N}
 N^2 J(N) J'(N) e^{\frac{2N J'(N)}{J(N)}} =k,
\end{equation} has a unique solution $N(k) \in \mathbb{R}_{>0}$ and $f(k)$ given by
\begin{equation}\label{growth f}
f(k)=  \frac{N(k) \sqrt{J'(N(k)) /J(N(k))}} {\sqrt{\log k}}, \end{equation}
satisfies Theorem \ref{christ}. When $k < k_0$, for simplicity we define $f(k)=f(k_0)$.
 \end{thm}
 
 Let us now draw some quick corollaries of this theorem. We say $h$ belongs to the Gevrey class $G^s$, $s\geq 1$, if $M(N)= N^{(s-1)N}$. Since this $M$ satisfies the conditions of the theorem for $s>1$ (but not for $s=1$) we immediately get $f(k)= \frac{\delta k^{\frac{1}{4s-2}}}{\sqrt{\log k}}$, for $\delta >0$ independent of $k$, thus we obtain the following improved upper bounds:
\begin{cor}  \label{Main} Assume $h \in G^s$ \footnote{Note that $\phi \in G^s$ if and only if $h=e^{-\phi} \in G^s$.}, $s > 1$. Then for all $k \in \mathbb N$, and all $z, w \in X$, we have
\begin{equation}\label{MainEstimate}
|B_k(z, w)|_{h^k} \leq 
\begin{cases}
C k^{n} e^{-c \, k d(z,w)^2},\qquad & d(z,w)\leq k^{- \frac{s-1}{2s-1}},\\
C k^{n} e^{-c \, k^{\frac{s}{2s-1}} d(z,w)}, \qquad & d(z,w)\geq k^{-\frac{s-1}{2s-1}}.
\end{cases}
\end{equation} 
where $c$ and $C$ are positive constants dependent only on $(L, h)$ and $X$.
\end{cor}
In the special case $s=1$, meaning when the metric $h$ is analytic, we obtain the following estimate which was stated without a proof in Remark 6.6 of \cite{Ch}:
\begin{cor}\label{Analytic} Assume $h$ is analytic. Then for all $k \in \mathbb N$ and $z,w \in X$, we have
\begin{equation} 
\left | B_k(z, w) \right |_{h^k} 
\leq  C k^n e^{- c k d(z,w)^2}.
\end{equation}
\end{cor}
Note that for this corollary we cannot directly use Theorem \ref{main} because in the analytic case $J(N)$ is bounded (and vice versa). We will give two proofs for this estimate, where one of them involves taking a uniform limit $s \to 1^+$ in Corollary \ref{Main}. To provide more examples, we can also consider Denjoy \cite{Denjoy} classes of quasi-analytic functions given by 
 $$ M(N)= (\log N)^N, \quad  (\log N)^N ( \log \log N)^N, \quad (\log N)^N ( \log \log N)^N ( \log \log \log N)^N, \quad  \cdots . $$ For instance if $M(N)= (\log N)^N$, from Theorem \ref{main} we get $f(k) \sim k^\frac12 (\log k) ^{-\frac32}$, which by plugging into Theorem \ref{christ} we obtain:
 \begin{cor} Assume $h \in \mathcal C_M$ \footnote{Again, $\phi$ belongs to this particular class if and only if $h=e^{-\phi}$ does.} with $M(N)= (\log N)^N$. Then for all $k \geq 2$ and all $z,w \in X$,
 \begin{equation}
|B_k(z,w)|_{h^k} \leq 
\begin{cases}
C k^{n} e^{-c \, k d(z,w)^2}, \quad & d(z,w)\leq \frac{1}{\log k},\\
C k^{n} e^{-\, \frac{ck}{\log k} \,\, d(z,w)}, \quad & d(z,w) \geq \frac{1}{\log k}.
\end{cases}
\end{equation}
\end{cor}
We note that in Theorem \ref{christ}, the maximum rate of growth of the function $f(k)$ is
$\frac{k^\frac12}{\sqrt{\log k}}$, which happens when the metric $h$ is analytic, but there is no minimum rate of growth for $f(k)$ as it can be arbitrary small by choosing the majorant $M$ in Theorem \ref{main} to have arbitrary large growth rate. Also, as we see all of the above estimates are much better than the following Agmon type estimates (see \cite{Bern, Ch1, Del, Lindholm, MaMaAgmon}) valid for all smooth metrics $h$ 
\begin{equation}\label{Agmon} \left | B_k(z,w) \right |_{h^k}  \leq C k^n e^{- c \sqrt{k} d(z,w)}.\end{equation}

We must emphasize that  on the diagonal and in certain shrinking neighborhoods of the diagonal we have more precise information on the Bergman kernel. Zelditch \cite{Ze1} and Catlin \cite{Ca} proved if $h$ is $C^\infty$ then on the diagonal $z=w$, the Bergman kernel accepts a complete asymptotic expansion of the form
\begin{equation}\label{ZC}
|B_k(z,z) |_{h^k}= K_k(z,z)e^{-k\phi(z)}\sim\frac{k^n}{\pi^n}\left(b_0(z,\bar{z})+\frac{b_1(z,\bar{z})}{k}+\frac{b_2(z,\bar{z})}{k^2}+\cdots\right).
\end{equation}
\textit{Very near the diagonal}, i.e. in a $\frac{1}{\sqrt{k}}$-neighborhood of the diagonal, one has a scaling asymptotic expansion for the Bergman kernel (see \cite{ShZe, MaMaBook, MaMaOff, LuSh, HKSX}).  In fact given any $\gamma >0$, for $d(z, w) \leq \gamma \sqrt{\frac{\log k}{{k}}}$ one has (see \cite{ShZe})
\begin{equation} \label{Shrinking} 
\left | B_k(z, w) \right |_{h^k} = \frac{k^n}{\pi^n} e^{- \frac{k D(z, w)}{2}} \left (1 + O\left(\frac{1}{k}\right) \right ). 
\end{equation}
Here $D(z, w)$ is Calabi's diastasis function \cite{Cal} which is controlled from above and below by $d^2(z, w)$ and defined by
\begin{equation}\label{Diastasis}
D(z, w)= \phi(z) + \phi(w) - \psi(z, \bar w) - \psi(w, \bar z),
\end{equation}	
where $\psi(z,\bar{w})$ is an (almost) holomorphic extension of the \k potential $\phi(z)$.  Furthermore, we expect this asymptotic property to hold for $\gamma=f(k)$:
\begin{conj} Assume $\partial ^2 \log h \in \mathcal C_M$ where $M$ satisfies the conditions of Theorem \ref{main} and $f(k)$ is given by \eqref{growth f}. Then the asymptotic \eqref{Shrinking} holds for $d(z,w) \leq f(k) \sqrt{\frac{\log k}{k}}$. In the analytic case \eqref{Shrinking} holds for $d(z,w) <\delta$ for some $\delta >0$ independent of $k$\footnote{In fact, in the analytic ase, this is a conjecture of Zelditch \cite{Zeemail}.}. 
\end{conj}
Related to this conjecture, in \cite{HLXanalytic} it is shown that if $h$ is analytic then \eqref{Shrinking} holds for $d(z,w) \leq C k^{-\frac14}$. Under the weaker assumption $h \in G^s$, for $s > 1$, it is proved in \cite{HLXgevrey} that \eqref{Shrinking} is valid for $d(z,w) \leq C_\eps k^{-\frac{1}{2}+\frac{1}{4s+4\eps}}$ for any $\eps>0$. 

\subsection{Organization of the paper} 
Our work is greatly inspired by \cite{Ch2}. To obtain a quantitative version of the result of \cite{Ch2}, we follow their proof thoroughly and keep track of all constants involved. We also fill in some details of \cite{Ch2} such as the bootstrapping arguments and the CR structure involved in the analytic hypoellipticity lemma for the Kohn Laplacian. As in \cite{Ch2}, we study the Bergman kernel in three regions, namely $$d(z, w) \leq \gamma \sqrt{\tfrac{\log k}{k}}, \quad \gamma \sqrt{\tfrac{\log k}{k}} \leq d(z, w) \leq f(k) \sqrt{\tfrac{\log k}{k}}, \quad \text{and} \quad d(z, w) \geq f(k) \sqrt{\tfrac{\log k}{k}}.$$ As we discussed above in fact in the region $d(z, w) \leq \gamma \sqrt{\tfrac{\log k}{k}}$ we have an asymptotic expansion for the Bergman kernel (see \cite{ShZe}). For completeness,  in Subsection \ref{verynear} we will provide a proof for this using \cite{BBS}.  In the regions $$ \gamma \sqrt{\tfrac{\log k}{k}} \leq d(z, w) \leq f(k) \sqrt{\tfrac{\log k}{k}}, \quad \text{and} \quad d(z, w) \geq f(k) \sqrt{\tfrac{\log k}{k}}, $$ which we will call \textit{near diagonal} and \textit{far from diagonal} respectively,  we study the Bergman kernel via studying the Green kernel of the Hodge-Kodaira Laplacian as in \cite{Ch2}. In fact far from diagonal estimates follow from near diagonal estimates by an iterative argument involving Neumann series exactly as it is done in \cite{Ch2}. One of the main ingredients in the proof of near diagonal estimates is our Lemma \ref{improved estimates}, which is new relative to \cite{Ch2}, and is the counterpart of the analytic hypoellipticity Lemma \ref{lemma hypoanalyticity} when the condition of analyticity of $h$ is replaced by $\partial^2 \log h$ belonging to the class $\mathcal C_M$. The function $f(k)$, which is responsible for the rate of the decay of the Green kernel (or the Bergman kernel) is obtained in this step and it is extracted from an interesting optimization problem where the convexity of $\log (M)$ plays a key role.  

Section \ref{Background}, provides background on the relation of the Bergman and Green kernels, a priori elliptic estimates, analytic hypoellipticity of the Kohn Laplacian and its underlying CR structure. Section \ref{near} involves proving the near diagonal estimates. The optimization problem for $f(k)$ is studied in Subsection \ref{f(k)}.  Subsection \ref{verynear}, gives a proof of the \textit{very near diagonal} asymptotic expansion \ref{Shrinking}. Finally, Section \ref{far} gives a proof for far from diagonal estimates using the near diagonal estimates.


\section{Background Materials}\label{Background}

\subsection{The relation between the Bergman kernel and the Green Kernel}
In this section we will introduce the Green kernel associated to the line bundle $L^k$ and explain its relation with the Bergman kernel.
 
Let 
\begin{equation*}
	\dbar: \Gamma(X, \Lambda^{0,q}(L^k))\rightarrow \Gamma(X,\Lambda^{0,q+1}(L^k))
\end{equation*}
be the usual Dolbeault operator. Denote by $\dstar$ the formal adjoint, with respect to the $L^2$ inner product induced by the Hermitian metric $h$ and \k metric $\omega$. Then we have the Hodge-Kodaira Laplacian 
\begin{equation*}
	\Box_{k}=\dbar\dstar+\dstar\dbar: \Gamma(X,\Lambda^{0,1}(L^k))\rightarrow \Gamma(X,\Lambda^{0,1}(L^k)).
\end{equation*}
Since the line bundle $(L,h)$ is positive, by Weitzenb\"ock formula (see \eqref{Weitzenbock formula}), there exists some positive constant $c>0$ such that
\begin{equation*}
	\left(\Box_{k}f, f\right)\geq ck \|f\|_{L^2}^2, \quad \mbox{ for any } f\in \Gamma(X,\Lambda^{0,1}(L^k)).
\end{equation*} 
By this lower bound and the fact that $\Box_{k}$ is formally self-adjoint, we have the Green operator 
\begin{equation*}
	G=\Box_{k}^{-1}: L^2(X,\Lambda^{0,1}(L^k))\rightarrow L^2(X,\Lambda^{0,1}(L^k)).
\end{equation*}
Let $G_{k}(z,w)$ be the distribution kernel of $G_{k}$, called the Green kernel. As $\Box_{k}$ is formally self-adjoint, we have the Hermitian symmetry
\begin{equation*}
	\oo{G_{k}(z,w)}=G_{k}(w,z).
\end{equation*}

For any section $f\in \Gamma(X,L^k)$, using H\"ormander's $L^2$ estimates, $u=\dstar G_{k}\dbar f$ is the solution of 
\begin{equation*}
	\dbar u=\dbar f,
\end{equation*}
with the minimal $L^2$ norm and $u \perp \ker \dbar$. Therefore, in terms of the Bergman projection $\Pi_k$, we obtain
\begin{equation*}
	\Pi_k f=f-\dstar G_{k}\dbar f.
\end{equation*}
Writing this into integrals, we get
\begin{align*}
	\int_M f(w)B_k(z,w)dV_{w}=f(z)-\dstar\int_M \dbar f(w) G_{k}(z,w)dV_w.
\end{align*} 
By integrating by parts and comparing the kernels, we get
\begin{equation}\label{Bergman and Green Kernel}
	B_{k}(z,w)=\oo{\partial}^*_{k,z} {\partial}^*_{k,w} G_k(z,w), \quad \mbox{ for any } z\neq w\in X.
\end{equation}
Here the subindex $z$ in $\oo{\partial}^*_{k,z}$ and $w$ in ${\partial}^*_{k,w}$ indicate which variable the operators acts on.

\subsection{An equivalent framework without weights}
In this section, we introduce an equivalent framework of the $L^2$ space of $L^k$ valued holomorphic sections. Though this frame only works in a local coordinate chart, it is convenient since the norms are not involved with any $k$-dependent weight.  

Given a local coordinate chart $U$, let $\phi$ be a \k potential. We define the linear map $\mathcal{I}: L^2(U,\Lambda^{0,q}(L^{k}))\rightarrow L^2(U,\Lambda^{0,q})$ by 
\begin{equation*}
u:=\mathcal{I} f={e^{-\tfrac{k}{2}\phi}}f.
\end{equation*}
Note $\int_U |f|^2e^{-k\phi}dV=\int_U |u|^2dV$. $\mathcal{I}$ is therefore a linear isometry, by which we can identify these two $L^2$ spaces. Let us carry out the corresponding operators of $\dbar,\dstar, \Box_{k}$ on $L^2(U,\Lambda^{0,1})$ via this identification. Define $\Dbar: \Gamma(U,\Lambda^{0,q})\rightarrow \Gamma(U,\Lambda^{0,q+1})$ as
\begin{align*}
\Dbar=\mathcal{I}^{-1}\circ \dbar\circ \mathcal{I}=e^{-\tfrac{k}{2}\phi}\,\dbar \,e^{\tfrac{k}{2}\phi}=\dbar+\tfrac{k}{2}\bdel\phi\wedge.
\end{align*}
Then the adjoint operator $\Dstar: \Gamma(U,\Lambda^{0,q+1})\rightarrow \Gamma(U,\Lambda^{0,q})$ becomes
\begin{align*}
\Dstar=\mathcal{I}^{-1}\circ \dstar\circ \mathcal{I}=e^{-\tfrac{k}{2}\phi}\,\dstar \,e^{\tfrac{k}{2}\phi},
\end{align*}
and we can define the Laplace operator $\Delta_{k}: \Gamma(U,\Lambda^{0,1})\rightarrow \Gamma(U,\Lambda^{0,1})$ by
\begin{align*}
\Delta_k=\Dbar\Dstar+\Dstar\Dbar=\mathcal{I}^{-1}\circ \Box_{k}\circ \mathcal{I}=e^{-\tfrac{k}{2}\phi}\,\Box_{k} \,e^{\tfrac{k}{2}\phi}.
\end{align*}
Similarly, the counterpart of the Green kernel $G_{k}(z,w)$ is 
\begin{equation}\label{Green Kernel}
	\mathcal{G}_{k}(z,w)=G_{k}(z,w) e^{-\frac{k}{2}\phi(z)} e^{-\frac{k}{2}\phi(w)},
\end{equation}
which represents a fundamental solution of $\Delta_{k}$ with a pole singularity at $w$. This is a section in $\Gamma(U\times U, \pi_1^*\Lambda^{0,1}\otimes \pi_2^*\Lambda^{0,1})$, where $\pi_1, \pi_2: U\times U\rightarrow U$ are projections to the first and second components respectively. In particular, $\mathcal{G}_{k}(z,w)$ is Hermitian symmetric in $z, w$ and if we take the norm, then
\begin{equation*}
|\mathcal{G}_{k}(z,w)|=|G_{k}(z,w)|_{h^{k}}.
\end{equation*}
Indeed, $L^2(U,\Lambda^{0,1})$ and $L^2(U,\Lambda^{0,1}(L^k))$ are equivalent because of the identification $\mathcal{I}$. We will work in the space $L^2(U,\Lambda^{0,1})$, where there is no weight in the inner product, while all the operators are twisted by $e^{\tfrac{k}{2}\phi}$. Note that we only need $k\geq 1$, which need not to be an integer in this unweighted framework.

\subsection{Elliptic Regularity}
In this section, we will recall some results from the regularity theory of the elliptic systems. For more details, we refer the readers to \cite{GiMa}.

The operators $\Box_{k}$ and $\Delta_{k}$ are both elliptic operators. For simplicity, we only state the results for $\Delta_{k}$ here and similar results also hold for $\Box_{k}$. In a local coordinate chart $U=B(0,1)$, we can write out the operator $\Delta_{k}$. For any $u\in \Gamma(U,\Lambda^{0,1})$,
\begin{equation}\label{laplace in local coordinates}
\left(\Delta_{k} u\right)_{\bar{s}}= -\partial_i\left(g^{i\bar{j}}\partial_{\bar{j}}u_{\bar{s}}\right)+\left(a_{\bar{s}}^{i\bar{j}}k+b_{\bar{s}}^{i\bar{j}}\right) \partial_iu_{\bar{j}}+\left(a_{\bar{s}}^{\bar{i}\bar{j}}k+b_{\bar{s}}^{\bar{i}\bar{j}}\right) \partial_{\bar{i}}u_{\bar{j}}+\left( c^{\bar{i}}_{\bar{s}}k^2+ d^{\bar{i}}_{\bar{s}} k+e^{\bar{i}}_{\bar{s}}\right) u_{\bar{i}},
\end{equation}
where we have used Einstein summation convention, and denoted
\begin{equation*}
	\Delta_{k}u=(\Delta_{k}u)_{\bar{s}}d\oo{z}^s, \quad u=u_{\bar{j}}d\oo{z}^j.
\end{equation*}
All the coefficients $a^{i\bar{j}}_{\bar{s}}, a^{\bar{i}\bar{j}}_{\bar{s}},b^{i\bar{j}}_{\bar{s}}, b^{\bar{i}\bar{j}}_{\bar{s}}, c^{\bar{i}}_{\bar{s}}, d^{\bar{i}}_{\bar{s}}$ and $e^{\bar{i}}_{\bar{s}}$ are polynomials of $g$, ${\phi}$, $g^{-1}$, and their derivatives (up to second order). For simplicity, we denote the above identity as 
\begin{equation*}
\Delta_{k} u= -\partial_i\left(g^{i\bar{j}}\partial_{\bar{j}}u\right)+(ak+b)*\nabla u+( ck^2+dk+e)* u.
\end{equation*}
Here $u$ is identified with the vector $(u_1, u_2,\cdots, u_n)$, $\nabla$ is the gradient operator, and $*$ denotes certain algebraic operations. Though $g^{-1}$, $a,b,c, d, e$ are smooth and bounded uniformly for any $C^m$ norm, $k$ could be arbitrarily large. In order to apply the elliptic estimates, we  need to make all the coefficients bounded uniformly by rescaling. For $R\in (0,1]$, define $\widetilde{g}(z)=g(Rz)$, $\widetilde{a}(z)=Ra(Rz)$, $\widetilde{b}(z)=Rb(Rz)$, $\widetilde{c}(z)=R^2c(Rz)$, $\widetilde{d}(z)=R^2d(Rz)$, $\widetilde{e}(z)=R^2e(Rz)$ and $\widetilde{u}(z)=u(Rz)$. Then  
\begin{equation}\label{rescaled equation}
-\partial_i\left(\widetilde{g}^{^{i\bar{j}}}\partial_{\bar{j}}\widetilde{u}\right)+\left(\widetilde{a}k+\widetilde{b}\right)*\nabla \widetilde{u}+\left(\widetilde{c}k^2+\widetilde{d}k+\widetilde{e}\right)* \widetilde{u}=R^2(\Delta_{k}u)(Rz).
\end{equation}  
When $R\leq \frac{1}{k}\leq 1$, all the coefficients are uniformly bounded for any $C^m$ norm. Let $\eta$ be a smooth cut-off function such that $\supp\eta \in B(0,\frac{3}{4}R)$ and $\eta=1$ in $B(0,\frac{1}{2}R)$. And let $\widetilde{\eta}(z)=\eta(Rz)$. By using the Caccioppoli inequality (see Theorem 4.11 in \cite{GiMa}), for any $m\in \mathbb{N}$, 
\begin{align*}
\|\widetilde{\eta}\widetilde{u}\|_{H^{m+1}(B(0,1))}\leq C\left(\|\widetilde{u}\|_{L^2(B(0,1))}+\|\Delta_{k}\widetilde{u}\|_{H^m(B(0,1))}\right),
\end{align*}
where $C$ only depends on $m$, the lower bound of $g^{-1}$ and the $C^{m+1}$ norm of coefficients $g^{-1},a,b,c,d,e$.
We change $\widetilde{u}$ back to $u$ and obtain the following interior estimates.
\begin{equation}\label{Interior Estimates}
\|u\|_{H^{m+1}(B(0,R/2))}\leq C \left(\tfrac{1}{R^{m+1}}\|u\|_{L^2(B(0,R))}+\tfrac{1}{R^{m-1}}\|\Delta_{k}u\|_{H^m(B(0,R))}\right), \quad \mbox{for any } R\leq \tfrac{1}{k}.
\end{equation}

\subsection{A priori upper bounds of the Green kernel}
In this section, we will prove an a priori upper bound of the Green kernel $\mathcal{G}_{k}$.
\begin{lemma}\label{Apriori Estimates}
	Given any point $p\in X$, there exists a coordinate chart $U$ containing $p$ and constants $C,a>0$, such that whenever $z, w\in U$ and $z\neq w$, we have
	\begin{equation}\label{eq Apriori Estimates}
	\left|\mathcal{G}_k(z,w)\right|+|\nabla_z \mathcal{G}_k(z,w)|+|\nabla^2_z \mathcal{G}_k(z,w)|\leq C\left(k+|z-w|^{-1}\right)^a,
	\end{equation}
	where $\nabla$ denotes the gradient.
\end{lemma}

This lemma is from \cite{Ch2} and here we fill in the details of the proof. Admittedly, sharper upper bounds on $\mathcal{G}_{k}(z,w)$ can be proved but this one is sufficient for our purpose. The proof is divided into two steps. First, we can rather easily obtain certain upper bound on the operator $\mathcal{G}_{k}$, whose distribution kernel is $\mathcal{G}_{k}(\cdot, \cdot)$. Second,  the operator bound can be improved to the pointwise bound in \eqref{eq Apriori Estimates} by the interior estimates \eqref{Interior Estimates}. This is a standard bootstrapping argument and will be used again in later sections. 
       
\begin{proof}
For $\Box_{k}: L^2(X,\Lambda^{0,1}(L^{k}))\rightarrow  L^2(X,\Lambda^{0,1}(L^{k}))$, Weitzenb\"ock formula tells that
\begin{equation}\label{Weitzenbock formula}
\Box_{k}=\oo{\nabla}^*\oo{\nabla}+\Ric TM+k\Ric L,
\end{equation}
where $\nabla$ is the Chern connection on the bundle $L^{k}$ coupled with the Levi-Civita connection on $\Lambda^{0,1}$. When $k$ is sufficiently large, $\Box_{k}$ is bounded below by $\tfrac{k}{2}$, whence $G_{k}$ has the operator bound $\|G_{k}\|_{L^2\rightarrow L^2}\leq \frac{2}{k}$. Let $V$ be a coordinate chart containing $p$. For simplicity, we can assume $V=B(0,2)$. Take $U=B(0,1)$ and for any $z,w\in U$, by compositing with inclusion and restriction maps, we can regard $G_{k}$ as a linear operator
\begin{equation*}
	G_{k}: L^2(B(w,\tfrac{1}{4}d(z,w)),\Lambda^{0,1}(L^{k}))
	\rightarrow L^2(B(z,\tfrac{1}{4}d(z,w)),\Lambda^{0,1}(L^{k})),
\end{equation*} 
with the operator bound $\|G_{k}\|\leq \tfrac{2}{k}$.  If we let
\begin{equation*}
	\mathcal{G}_{k}:L^2(B(w,\tfrac{1}{4}d(z,w)),\Lambda^{0,1})\rightarrow L^2(B(w,\tfrac{1}{4}d(z,w)),\Lambda^{0,1})
\end{equation*} 
be the linear operator with distribution kernel $\mathcal{G}_{k}(\cdot, \cdot)$, then $\mathcal{G}_{k}$ shares the same operator bound $\|\mathcal{G}_{k}\|\leq \tfrac{2}{k}$ by \eqref{Green Kernel}. That is, for any $u\in L^2(B(w,\frac{1}{4}d(z,w)),\Lambda^{0,1})$ with $\supp u\subseteq B(w,\frac{1}{4}d(z,w))$, we have
	\begin{equation*}
		\|\mathcal{G}_{k}u\|_{L^2(B(z,\frac{1}{4}d(z,w)))}
		\leq \tfrac{2}{k} \|u\|_{L^2(B(w,\frac{1}{4}d(z,w)))}.
	\end{equation*}
Since $\Delta_{k}\mathcal{G}_{k}u=u$, which vanishes in $B(z,\tfrac{1}{4}d(z,w))$, if we apply the elliptic estimates \eqref{Interior Estimates} to $\mathcal{G}_{k}u$ on $B(z,R)$ with $R=\min\{\tfrac{1}{k},\tfrac{1}{4}d(z,w)\}$, then for any $m\in \mathbb{N}$, 
	\begin{equation*}
	\|\eta \mathcal{G}_{k}u\|_{H^m(B(z,R))}\leq CR^{-m}\|\mathcal{G}_{k}u\|_{L^2(B(z,R))}
	\leq C\left(k+|z-w|^{-1}\right)^m \|u\|_{L^2(B(w,\frac{1}{4}d(z,w)))},
	\end{equation*} 
	where $C$ is a constant depending only on $C^{m+3}$ norm of the \k potential $\phi$ and the lower bound of $\sqrt{-1}\del\bdel\phi$. Taking $m=n+1$ and applying the Sobolev embedding theorem,
	\begin{equation*}
	|\mathcal{G}_{k}u(z)|
	\leq 
	C\left(k+|z-w|^{-1}\right)^{n+1} \|u\|_{L^2(B(w,\frac{1}{4}d(z,w))}.
	\end{equation*}  
	Therefore,
	\begin{equation*}
	\|\mathcal{G}_{k}(z,\cdot)\|_{L^2(B(w,\frac{1}{4}d(z,w)))}
	=\sup_{u\in L^2(B(w,\frac{1}{4}d(z,w))} \frac{\left|\int u(\zeta)\mathcal{G}_{k}(z,\zeta)d\zeta\right|}{\|u\|_{L^2(B(w,\frac{1}{4}d(z,w))}}\leq C\left(k+|z-w|^{-1}\right)^{n+1}.
	\end{equation*}
	Since for any $\zeta\in B(w,\frac{1}{4}d(z,w))$,
	\begin{equation*}
	\Delta_{k}\mathcal{G}_{k}(z,\zeta)=0,
	\end{equation*}
	when $\Delta_{k}$ acts on the second component, again by applying the elliptic estimates \eqref{Interior Estimates} to $\mathcal{G}_{k}(z,\cdot)$ on $B(w,R)$ with $R=\min\{\tfrac{1}{k},\tfrac{1}{4}d(z,w)\}$,  and using the Sobolev embedding theorem, we have
	\begin{equation*}
	|\mathcal{G}_{k}(z,w)|_{C^2}
	\leq 
	C\left(k+|z-w|^{-1}\right)^{n+3}
	\|\mathcal{G}_{k}(z,\cdot)\|_{L^2(B(w,\frac{1}{4}d(z,w)))}
	\leq 
	C\left(k+|z-w|^{-1}\right)^{2n+4},
	\end{equation*}
	where the pointwise $C^2$ norm is taken on the second component $w$. Note the constant $C$ only depends on the dimension $n$ and the $C^{n+6}$ norm of \k potential $\phi$ and the lower bound of $\sqrt{-1}\del\bdel\phi$. Hence, we can make $C$ uniform for all $z,w\in U$. The result follows by the symmetry $\mathcal{G}(z,w)=\oo{\mathcal{G}(w,z)}$ and taking $a=2n+4$. 
\end{proof}
 
\subsection{Analytic hypoellipticity of the Kohn Laplacian} 
In this section, we work in the unweighted framework introduced above. Let $B\subseteq \mathbb{C}^n$ be an open ball and let $\widetilde{B}\Subset B$ be a relatively compact subset. When metrics are real analytic, M. Christ has proved the following estimates in \cite{Ch2}(see Lemma 7) by using the analytic hypoellipticity of the Kohn Laplacian \cite{Ta2, Tr}. 
\begin{lemma}\label{lemma hypoanalyticity}
	Let the ball $B\subseteq \mathbb{C}^n$ be equipped with a real analytic Hermitian metric $g$. Let $L$ be a holomorphic line bundle over $B$, equipped with a positive real analytic Hermitian metric $h$. For any relatively compact $\widetilde{B}\Subset B$, there exist positive constant $C$ and $b$, such that for any solution $u\in \Gamma(B,\Lambda^{0,q})$ $(0< q < n)$ of $\Delta_{k}u=0$ on $B$, 
	\begin{equation}\label{hypoanalyticity}
	\|u\|_{L^{\infty}(\widetilde{B})}\leq Ce^{-bk}\|u\|_{L^{\infty}(B)}
	\end{equation}  
\end{lemma}  

\begin{rmk}\label{uniform hypoanalyticity}
	This lemma does not require the metric $g$ is polarized by the line bundle. We still denote $\phi=-\log h$, which is a \k potential in the polarized case. Positivity	of $h$ means that $\left(\phi_{i\bar{j}}\right)\geq cI$ for some $c>0$. The constant $C$ is a universal numerical constant. The constant $b$ depends on the following data:
	\begin{itemize}
		\item[(a)] the lower bound of $\sqrt{-1}\del\bdel \phi$, i.e.  a positive constant $c$ such that $\left(\phi_{i\bar{j}}\right)\geq cI$ in $B$;
		\item[(b)] the lower bound of $(g_{i\bar{j}})$, i.e.  a positive constant $c$ such that $\left(g_{i\bar{j}}\right)\geq cI$ in $B$;
		\item[(c)] the analyticity constant of $\phi$, i.e. a positive constant $C(\phi)$ such that
		\begin{equation*}
		\left|D_{(z,\bar{z})}^{\alpha}\phi(z) \right|\leq C(\phi)^{|\alpha|+1}\alpha!, \quad \mbox{ for any multi-index } \alpha \mbox{ and } z\in B; 
		\end{equation*}  
		\item[(d)] the analyticity constant of each entry of metric $g$;
		\item[(e)] the distance from $\widetilde{B}$ to $\del B$.
	\end{itemize}
\end{rmk}

For the completeness, we include a proof here for $n\geq 2$. To begin with, we endow $B\times \mathbb{R}$ with a strictly pseudoconvex CR structure (see \cite{ChSh} for more details).
Set the domain 
\begin{equation*}
	\Omega=\{(z,w)\in B\times \mathbb{C}: \Ima w>\tfrac{\phi(z)}{2}\}.
\end{equation*}
Then its boundary is 
\begin{equation*}
	\partial\Omega=\{(z,w)\in B\times \mathbb{C}: \Ima w=\tfrac{\phi(z)}{2}\},
\end{equation*}
and it defining function is $\rho(z,w)=\Ima w-\tfrac{\phi(z)}{2}$. A straightforward calculation shows that if we define
\begin{equation*}
	L_p:=\frac{\partial}{\partial z_p}+i\phi_p\frac{\partial}{\partial w},\quad \mbox{for } 1\leq p\leq n,
\end{equation*} 
then $\{L_p\}$ form a global basis for the space of $T^{1,0}\partial\Omega$. Clearly, 
\begin{equation*}
	[L_p,L_q]=0, \quad \mbox{for any } 1\leq p,q\leq n.
\end{equation*}
Thus $(\partial\Omega, T^{1,0}\partial\Omega)$ is a CR manifold. 

We can identify the boundary $\partial\Omega$ with $B\times \mathbb{R}$ via the diffeomorphism $\pi: \partial\Omega\rightarrow B\times \mathbb{R}$ defined as
\begin{equation*}
	\pi(z,t+i\tfrac{\phi(z)}{2})\rightarrow (z,t).
\end{equation*}
Therefore, a CR structure can be induced on $B\times \mathbb{R}$,  via $\pi$, that is,
\begin{equation*}
	\pi_*L_p=\pi_*\left(\tfrac{\partial}{\partial z_p}+i\phi_p\tfrac{\partial}{\partial w}\right)=\tfrac{\partial}{\partial z_p}+i\tfrac{\phi_p}{2}\tfrac{\partial}{\partial t},\quad \mbox{for } 1\leq p\leq n.
\end{equation*}
We still use $L_p$ for $\pi_*L_p$ for $1\leq p\leq n$ by abusing notations. If we denote $T=\tfrac{\partial}{\partial t}$, then 
\begin{equation*}
[L_p,\oo{L}_q]
=[\tfrac{\partial}{\partial z_p}+i\tfrac{\phi_p}{2}\tfrac{\partial}{\partial t},\tfrac{\partial}{\partial \oo{z}_q}-i\tfrac{\phi_{\bar{q}}}{2}\tfrac{\partial}{\partial t}]
=-i\phi_{p\bar{q}}T.
\end{equation*}
Since $\phi=-\log h$ is plurisubharmonic by the positivity of the Hermitian line bundle $(L,h)$, the CR structure on $B\times \mathbb{R}$ is strictly psedoconvex.

We will further endow the CR manifold $B\times \mathbb{R}$ with a compatible Hermitian structure. Note $L_p,\oo{L}_p$ for $1\leq p\leq n$ and $T$ form a global basis of the tangent bundle. We can directly define an Hermitian metric $h$ such that 
\begin{equation*}
h(L_p, L_q)=g_{p\bar{q}}, \quad 
h(L_p,\oo{L}_q)=h(L_p,T)=h(\oo{L}_p,T)=0,\quad
h(T,T)=1.
\end{equation*}
Set $\omega_p=dz_p$, $\oo{\omega}_p=d\oo{z}_p$ for $1\leq p\leq n$ and $\tau=dt-i\tfrac{\phi_p}{2}dz_p+i\tfrac{\phi_{\bar{p}}}{d\oo{z}_p}$, which form a dual basis to $L_p,\oo{L_p},T$. Then $h$ naturally induces a Hermitian metric on the cotangent bundle such that
\begin{equation}\label{CR metric}
h(\omega_p,\omega_q)=g^{p\bar{q}}.
\end{equation}
The induced volume form is therefore 
\begin{equation}\label{CR volume}
dV_{B\times\mathbb{R}}=\det g\tfrac{\sqrt{-1}}{2}dz_1\wedge d\oo{z}_1\wedge \tfrac{\sqrt{-1}}{2}dz_2\wedge d\oo{z}_2\cdots \wedge \tfrac{\sqrt{-1}}{2}dz_n\wedge d\oo{z}_n\wedge dt=dV_g\wedge dt.
\end{equation}
 
 We are now ready to prove Lemma \ref{lemma hypoanalyticity}. 
 \begin{proof}
 Set $U(z,t)=u(z)e^{ik t}$ as a $(0,q)$ form on $B\times \mathbb{R}$. Let $\dbarb$ be the Cauchy Riemann operator associated to the CR structure on $B\times\mathbb{R}$. Then
 \begin{align*}
 \dbarb U=\sum_{p=1}^{n} d\oo{z}_p\wedge \oo{L}_pU=e^{ik t}\left(\dbar u+\tfrac{k}{2}\dbar\phi\wedge u\right)=e^{ik t}\oo{D}_{k}e^{-ik t}U.
 \end{align*} 
 Since the metric and volume form on $B\times \mathbb{R}$ as shown in \eqref{CR metric} and \eqref{CR volume} are compatible with those of $(B,g)$, the formal adjoint $\dbarb^*$ satisfies
 \begin{equation*}
 \dbarb^*U=e^{ik t}\oo{D}_{k}^*e^{-ik t}U.
 \end{equation*}
 Therefore, 
 \begin{equation*}
 \Box_bU=e^{ik t}\Delta_{k}u=0.
 \end{equation*}
 Note the constructed $CR$ structure on $B\times \mathbb{R}$ is real analytic and strictly pseudoconvex, while the compatible Hermitian metric on $B\times \mathbb{R}$ is also real analytic.  By the analytic hypoellipticity of Kohn Laplacian $\Box_b$ for $n\geq 2$ (\cite{Ta2,Tr}), the solution $U$ is real analytic. What is more, there exists some positive constant $C$, which only depends on the data in Remark \ref{uniform hypoanalyticity}, such that
 \begin{equation*}
 \|D^{\alpha}U(z,t)\|_{L^\infty(\widetilde{B}\times \mathbb{R})}\leq  \|U\|_{L^\infty(B\times \mathbb{R})} C^{|\alpha|}\alpha!, \quad \mbox{for any multi-index } \alpha.
 \end{equation*}   
 In particular, 
 \begin{equation*}
 \|k^mu(z)\|_{L^\infty(\widetilde{B})}=\|D_t^mU\|_{L^\infty(\widetilde{B}\times \mathbb{R})}\leq C^mm!\|u\|_{L^\infty(B)}, \quad \mbox{for } m\geq 0.
 \end{equation*}
 Therefore,
 \begin{equation*}
  \|u\|_{L^\infty(\widetilde{B})}\leq \frac{C^mm!}{k^m}\|u\|_{L^\infty(B)}, \quad \mbox{for } m\geq 0.
 \end{equation*}
 Take $m= [\tfrac{k}{C}]$ and then by Stirling's approximation,
 \begin{equation*}
  \|u\|_{L^\infty(\widetilde{B})}
  \leq \frac{m!}{m^m}\|u\|_{L^\infty(B)}
  \leq e^{-m+1}m^{\tfrac{1}{2}}\|u\|_{L^\infty(B)}
  \leq C_1 e^{-bk}\|u\|_{L^\infty(B)},
 \end{equation*}
 where $C_1=e$ and $b=\frac{1}{2C}$. 
 \end{proof}
 
\section{Estimates of the Green kernel near the diagonal}\label{near}
Given $p\in X$, let $V$ be a coordinate chart containing $p$. Without the loss of generality, we can assume $V=B(p,3)\subseteq \mathbb{C}^n$. 
Recall that $M(x)$ is a $C^2$ function on $\mathbb{R}_{>0}$ such that $\log \left({M(x)}\right)$ is strictly convex. We define the following class of functions, the growth rate of whose Taylor coefficients are controlled by $M$.
\begin{defn}\label{Majorant}
		Suppose $u\in C^{\infty}(V)$. We say the Taylor coefficients of $u$ are majorized by $M$, if there exists some positive constant $A$ such that for any multi-index $\alpha \in \mathbb{Z}_{\geq 0}^{2n}$ with $|\alpha| \neq 0$ and any $z\in V$, 
  		\begin{equation}\label{eq majorant}
  		\left|\frac{{D}^{\alpha}u}{\alpha!}(z)\right|\leq A^{|\alpha|} {M( |\alpha|)}.
  		\end{equation}
		We shall use $\mathcal{C}_M(V)$ to denote the collection of all such smooth functions on $V$. We shall also say a family of functions is uniformly in $\mathcal{C}_M(V)$, if there exists a positive constant $A$ satisfying \eqref{eq majorant} for any function in the family.
\end{defn}

\begin{rmk}
It is not hard to verify that condition \eqref{eq majorant} is invariant under holomorphic coordinates change if we allow $A$ to depend on the coordinate chart and local coordinates. Therefore, we can similarly define the class $\mathcal{C}_M(X)$. 
\end{rmk}
 
Let $\phi$ be a \k potential on $V$ such that all the second order derivatives $\partial^2\phi\in \mathcal{C}_M(V)$.  Set $U=B(p,1)$ and we will prove the following estimates of $\mathcal{G}_{k}(z,w)$ for $z,w$ in a shrinking neighborhood in $U$. We first introduce some notations before stating the result.

Let $J(x)=M(x)^{1/x}$ and assume $J(x)$ is unbounded.  In Lemma \ref{G is going to infinity}, we will prove that $$x^2J(x)J'(x)e^{\frac{2xJ'(x)}{J(x)}}$$ is strictly increasing to infinity starting from some point $x_0>0$ and we let 
\begin{equation}\label{k0} k_0= x_0^2J(x_0)J'(x_0)e^{\frac{2x_0J'(x_0)}{J(x_0)}}. \end{equation}
Then for each integer $k \geq k_0$, the equation 
 \begin{equation*}
 N^2 J(N) J'(N) e^{\frac{2N J'(N)}{J(N)}} =k,
 \end{equation*} 
 has a unique solution $N(k) \in \mathbb{R}_{>0}$ and we define $f(k)$ as
\begin{equation*}
 f(k)=  \frac{N(k) \sqrt{J'(N(k)) /J(N(k))}} {\sqrt{\log k}},
\end{equation*}
and for $k < k_0$ we define $f(k)=f(k_0)$. By the following lemma, whose proof will be given later, we know that:
\begin{lemma}\label{f infinity}
	Then $f(k)$ is strictly increasing to infinity for $k \geq k_0$.
\end{lemma}

This section is devoted to proving the following theorem, which proves our near diagonal estimates with $f(k)$ defined above. 
 
\begin{thm}\label{Estimate of G in a shrinking neighborhood}
	Suppose the \k potential $\phi$ satisfies that $\partial^2\phi\in \mathcal{C}_M(V)$ where $M$ is strictly logarithmically convex and $M(x)^{\frac1x}$ is unbounded. Then there exist positive constants $a, b, C, \gamma, \kappa$, independent of $k$ so that for any $z,w\in U$ with $\gamma \sqrt{\frac{\log k}{k}} \leq |z-w|\leq f(k)\sqrt{\frac{\log k}{k}}$ and $k\geq \kappa$,
	\begin{equation*}
	|\mathcal{G}_{k}(z,w)|+|\nabla_{z}\mathcal{G}_{k}(z,w)| + |\nabla_{z} \nabla_{w}\mathcal{G}_{k}(z,w)|\leq C e^{-bk |z-w|^2}.
	\end{equation*} 
	When $\phi$ is analytic this estimate holds in the much larger neighborhood $\gamma\sqrt{\frac{\log k}{k}} \leq |z-w|\leq 1$ for any $k\geq \kappa$.  
\end{thm}

\begin{rmk}
	The constant $b$ depend on the following data:
	\begin{itemize}
		\item[(a)] the lower bound of $\sqrt{-1}\del\bdel \phi$, i.e.  a positive constant $c$ such that $\left(\phi_{i\bar{j}}\right)\geq cI$ in $B$;
		\item[(b)] the lower bound of $(g_{i\bar{j}})$, i.e.  a positive constant $c$ such that $\left(g_{i\bar{j}}\right)\geq cI$ in $B$;
		\item[(c)] the constant $A$ as in \eqref{eq majorant} for metric $g$ and $\partial^2\phi$.
	\end{itemize}
	And $C$ depends on $(a),(b),(c)$ and certain $C^m$ norm of $\phi$ and $g$ for some $m$ only depending on the dimension $n$.
\end{rmk}

The following lemma is the counterpart of Lemma \ref{lemma hypoanalyticity} for metrics in the class $\mathcal{C}_M$, which is a key step in proving the above theorem.  As in \cite{Ch2}, we will construct analytic metrics to approximate $g$ and $h$ in the $\mathcal{C}_M$ class and apply Lemma \ref{lemma hypoanalyticity} to the approximation metrics. As the constants in \eqref{hypoanalyticity} depend only on the data mentioned in Remark \ref{uniform hypoanalyticity}, which are uniform in a shrinking ball depending on $k$ for all the approximation metrics, we can prove \eqref{hypoanalyticity} for the $\mathcal{C}_M$ metrics in such a shrinking ball.   
\begin{lemma}\label{improved estimates}
	Let the ball $B=B(0,3)\subseteq \mathbb{C}^n$ be equipped with a smooth Hermitian metric $g\in \mathcal{C}_M(B)$. Let $L$ be a holomorphic line bundle over $B$, equipped with a positive smooth Hermitian metric $h$ such that all second order derivatives $\partial^2\log h\in \mathcal{C}_M(B)$. Then for each $\delta \in (0, 1)$, there exist positive constants $C$ and $b$, such that for any solution $u\in \Gamma(B,\Lambda^{0,q})$ $(0< q < n)$ of $\Delta_{k}u=0$ on $B$,
	\begin{equation}\label{Gevrey hypoellipticity}
	\|u\|_{L^\infty(B(z,\tfrac{1}{2} \delta f(k)\sqrt{\frac{\log k}{k}}))}\leq Ce^{-bf^2(k)\log k}\|u\|_{L^{\infty}(B(z, \delta f(k)\sqrt{\frac{\log k}{k}}))}, \quad \mbox{ for any } z\in B(0,1),
	\end{equation}  
	where $f$ is defined in \eqref{growth f}.
\end{lemma} 
Here $\partial^2\log h\in \mathcal C_M(B)$ means that for every local potential $\phi=-\log h$ on an open set $B$, there exists $A>0$ satisfying \eqref{eq majorant} for any multi-index $\alpha\in \mathbb{Z}_{\geq 0}^{2n}$ with $|\alpha|\neq 0$ and any $z\in B$,

\begin{rmk}\label{uniform improved estimates}
		This lemma does not require the metric $g$ is polarized by the line bundle. We still denote $\phi=-\log h$, which is a \k potential in the polarized case. The constants $C$ and $b$ depend on the same data as in the previous remark.
\end{rmk}

Before proving this lemma, we first show the application of this lemma in proving Theorem \ref{Estimate of G in a shrinking neighborhood}.

\begin{proof}[Proof of Theorem \ref{Estimate of G in a shrinking neighborhood}]
	Given $w\in B(p,1)$, by using the Bochner coordinates at $w$, we can write the \k potential as $\phi(z)=|z|^2+O(|z|^4)=|z|^2\varphi(z)$ for some $\varphi\in C^{\infty}(B(w,2))$ and any $z\in B(w,2)$. Fix $r\geq \tfrac{2}{\sqrt{k}}$ and define a linear map $T: B(0,2)\subseteq \mathbb{C}^n\rightarrow U$ such that $T(\zeta)=r\zeta$. Pulling back the operator $\Dbar=\oo{D}_{k,\phi}$, we obtain
	\begin{equation}\label{Dbar operator}
	\oo{D}_{k,\phi(r\zeta)}^\dagger=\dbar+\tfrac{k}{2} \dbar\left(\phi(r\zeta)\right)\wedge=\dbar+\tfrac{\widetilde{k}}{2} \dbar\left(\tfrac{\phi(r\zeta)}{r^2}\right)\wedge=\oo{D}^\dagger_{\widetilde{k}, \tfrac{\phi(r\zeta)}{r^2}},
	\end{equation}
	where $\dbar$ always acts on the $\zeta$ variable and $$\widetilde{k}=k r^2>\log k.$$ Associated to the $L^2$ inner product induced by the metric $g(r\zeta)$ over $B(0,2)$, we can also define the adjoint of $\oo{D}^{\dagger}_{k,\phi(r\zeta)}$ and the Laplace operator $\Delta^{\dagger}_{k,\phi(r\zeta)}$. Note $\tfrac{\phi(r\zeta)}{r^2}=|\zeta|^2\varphi(r\zeta)$ as a function of $\zeta$ belongs to $C^{\infty}(B(0,2))$ for any $0<r\leq 1$. What is more, the following data are also uniform for $0<r\leq 1$.
	\begin{itemize}
		\item[(a)] $\del\bdel \tfrac{\phi(r\zeta)}{r^2}=\left(\del\bdel \phi\right)(r\zeta)$ is bounded below uniformly for $0<r\leq 1$.
		\item[(b)] $g(r\zeta)$ is bounded below uniformly for $0<r\leq 1$.
		\item[(c)] $\partial^2\tfrac{\phi(r\zeta)}{r^2}=\partial^2\phi(r\zeta)$ is in $\mathcal{C}_M$ uniformly for $0<r\leq 1$. This is to say,  there exists a positive constant $A$ satisfying \eqref{eq majorant} for all functions $\partial^2\tfrac{\phi(r\zeta)}{r^2}$ with $0<r\leq 1$.
		\item[(d)] The metric $g(r\zeta)$ is in $\mathcal{C}_M$ uniformly for $0< r\leq 1$.
	\end{itemize}
	Pulling back the Green kernel $\mathcal{G}_{k}$ via $T$ and using \eqref{Dbar operator}, for any nonzero $\zeta \in B(0,2)$ we have
	\begin{align}\label{Laplace on unit ball}
	\Delta^\dag_{\widetilde{k},\tfrac{\phi(r\zeta)}{r^2}}\mathcal{G}_{k}(r\zeta,0)
	=\Delta^\dag_{k,\phi(r\zeta)}\mathcal{G}_{k}(r\zeta,0)=T^*\left(\Delta_{k}\mathcal{G}_{k}\right)(\zeta,0)=0.
	\end{align}
	By taking $\delta=\frac{1}{4}$ in Lemma \ref{improved estimates} , for any $\zeta$ with $\frac{1}{3}f(\widetilde{k})\sqrt{\frac{\log\widetilde{k}}{\widetilde{k}}}\leq |\zeta|\leq \frac{4}{3}f(\widetilde{k})\frac{\sqrt{\log\widetilde{k}}}{\widetilde{k}}$, we have
	\begin{align*}
	|\mathcal{G}_{k}(r\zeta,0)|
	\leq& Ce^{-bf^2(\widetilde{k})\log\widetilde{k}}\|\mathcal{G}_{k}(r\,\cdot,0)\|_{L^{\infty}(B(\zeta,\tfrac{1}{4}f(\widetilde{k})\sqrt{\frac{\log\widetilde{k}}{\widetilde{k}}}))}.
	\end{align*}
	Note that whenever $\xi\in B(\zeta,\tfrac{1}{4}f(\widetilde{k})\sqrt{\frac{\log\widetilde{k}}{\widetilde{k}}})$, one has $|\xi|\geq \tfrac{1}{12}f(\widetilde{k})\sqrt{\frac{\log\widetilde{k}}{\widetilde{k}}}$. Since $f(k)$ is bounded from below by a positive constant by Lemma \ref{f infinity}, we have $|\xi|\geq ck^{-\tfrac{1}{2}}$ for some positive constant $c$. Now for any $\tfrac{2}{\sqrt{k}}\leq r\leq 1$, the a priori estimates on $\mathcal{G}_{k}$ in Lemma \ref{Apriori Estimates} imply
	\begin{align*}
	|\mathcal{G}_{k}(r\zeta,0)|
	\leq 
	Ck^ae^{-bf^2(\widetilde{k})\log\widetilde{k}}.
	\end{align*} 
	The constants $C$ and $b$ only depend on the positivity of $\sqrt{-1}\partial\bar{\partial}\phi$, the constant $A$ in \eqref{eq majorant} for $\partial^2\phi$, and $n$. The constant $a$ only depends on the dimension $n$.
	Since for any $\zeta\neq 0$, 
	\begin{equation*}
	\Delta_{\widetilde{k},\tfrac{\phi(r\zeta)}{r^2}}^{\dagger}\mathcal{G}_{k}(r\zeta,0)=0,
	\end{equation*}
	where $\Delta_{\widetilde{k},\tfrac{\phi(r\zeta)}{r^2}}^{\dagger}$ acts on the first variable $\zeta$, by using the interior estimates \eqref{Interior Estimates} with $R=\tfrac{1}{6\widetilde{k}}$, and Sobolev embedding theorem, we obtain
	\begin{align*}
	\|\mathcal{G}_{k}(r\zeta,0)\|_{C^1(A(0;\frac{1}{2}f(\widetilde{k})(\sqrt{\frac{\log\widetilde{k}}{\widetilde{k}}}),f(\widetilde{k})(\sqrt{\frac{\log\widetilde{k}}{\widetilde{k}}}))}
	\leq& Ck^a \|\mathcal{G}_{k}(r\zeta,0)\|_{L^\infty(A(0;\frac{1}{3}f(\widetilde{k})(\sqrt{\frac{\log\widetilde{k}}{\widetilde{k}}}),\frac{4}{3}f(\widetilde{k})(\sqrt{\frac{\log\widetilde{k}}{\widetilde{k}}})))}
	\\
	\leq& C k^a e^{-bf^2(\widetilde{k})\log\widetilde{k}}
	\\
	\leq& C k^a e^{-\frac{b}{4}k|r\zeta|^2}	
	\end{align*}
	In each line, $C$ and $a$ may be renamed by new constants, but they only depend on the positivity of $\sqrt{-1}\partial\bar{\partial}\phi$, the constant $A$ in \eqref{eq majorant} for $\partial^2\phi$, and $n$. And $a$ still only depends on $n$. Therefore, uniformly for any $w\in U$, any $r$ with $\tfrac{2}{\sqrt{k}}\leq r\leq 1$, and any $z,w\in U$ satisfying $$\frac{r}{2}f(k r^2)\sqrt{\tfrac{\log(k r^2)}{kr^2}}\leq |z-w|\leq rf(k r^2)\sqrt{\tfrac{\log(k r^2)}{kr^2}},$$ we have
	\begin{equation*}
	|\mathcal{G}_{k}(z,w)|+|\nabla_{z}\mathcal{G}_{k}(z,w)| \leq C k^a e^{-bk |z-w|^2}.
	\end{equation*} 
	In particular, if we vary $r$ within the interval $[2k^{-1/2},1]$, we have for $z,w\in U$ with $\sqrt{\frac{\log k}{k}} \leq |z-w|\leq f(k)\sqrt{\frac{\log k}{k}}$,
	\begin{equation*}
	|\mathcal{G}_{k}(z,w)|+|\nabla_{z}\mathcal{G}_{k}(z,w)|\leq Ck^a e^{-bk |z-w|^2}.
	\end{equation*} 
	
	Since $\Delta_k \nabla_z\mathcal{G}_k(z,w)=\Delta_k \nabla_z\mathcal{G}_k(z,w)=0$ if $\Delta_k$ is acting on $w$ variable, the estimate on $\nabla_z\nabla_{w}\mathcal{G}_{k}(z,w)$ follows by bootstrapping. By the relation \eqref{Bergman and Green Kernel}, this lemma immediately implies the Theorem \ref{main} for any points $z, w \in X$ with $\gamma\sqrt{\frac{\log k}{k}}\leq d(z,w)\leq f(k)\sqrt{\frac{\log k}{k}}$. Here $\gamma$ is a sufficiently large constant so that $k^a$ can be absorbed by $e^{-\tfrac{b}{2}kd(z,w)^2}$. This finishes the proof of Theorem \ref{Estimate of G in a shrinking neighborhood}, except it remains to prove it in the analytic case. 
\end{proof} 

\subsection{ Proof of Theorem \ref{Estimate of G in a shrinking neighborhood} in the analytic case.} We provide two proofs. The first proof \footnote{This proof was communicated to us by M. Christ \cite{christemail}.} is obvious from the above rescaling argument, the only change is that instead of Lemma \ref{improved estimates} which is specialized for the non-analytic cases, we use Lemma \ref{lemma hypoanalyticity} combined with the uniformity of the following data:
\begin{itemize}
		\item[(a)] $\del\bdel \tfrac{\phi(r\zeta)}{r^2}=\left(\del\bdel \phi\right)(r\zeta)$ is bounded below uniformly for $0<r\leq 1$.
		\item[(b)] $g(r\zeta)$ is bounded below uniformly for $0<r\leq 1$.
		\item[(c)] $\partial^2\tfrac{\phi(r\zeta)}{r^2}=\partial^2\phi(r\zeta)$ is uniformly analytic for $0<r\leq 1$.
		\item[(d)] The metric $g(r\zeta)$ is uniformly analytic for $0< r\leq 1$.
	\end{itemize}
The second proof involves taking the limit $s \to 1^+$ in Corollary \ref{Main}. Since we need to understand the constants involved in terms of $s$, we shall give this proof in Subsection \ref{2ndanalytic}. 

\subsection{Proof of Lemma \ref{improved estimates}}\label{sec estimates in shrinking neighborhood}
By translation, we can assume $z=0$ and work in $B(0,2)$. Take the Taylor expansion of $\phi$ at $0$, that is,
\begin{equation*}
	\phi(\zeta)=\phi(0)+\sum_{i=1}^{n}\phi_i(0)\zeta_i+\sum_{i=1}^{n}\phi_{\bar{i}}(0)\oo{\zeta}_i+\tfrac{1}{2}\sum_{i,j=1}^{n}\phi_{ij}(0)\zeta_i\zeta_j+\sum_{i,j=1}^{n}\phi_{i\bar{j}}(0)\zeta_i\oo{\zeta}_j+\tfrac{1}{2}\sum_{i,j=1}^{n}\phi_{\bar{i}\bar{j}}(0)\oo{\zeta}_i\oo{\zeta}_j+O(|\zeta|^3). 
\end{equation*}
Set
\begin{equation*}
	P=\phi(0)+\sum_{i=1}^{n}\phi_i(0)\zeta_i+\sum_{i=1}^{n}\phi_{\bar{i}}(0)\oo{\zeta}_i+\tfrac{1}{2}\sum_{i,j=1}^{n}\phi_{ij}(0)\zeta_i\zeta_j+\tfrac{1}{2}\sum_{i,j=1}^{n}\phi_{\bar{i}\bar{j}}(0)\oo{\zeta}_i\oo{\zeta}_j.
\end{equation*}
Let $Q$ be the real-valued harmonic conjugate of $P$, normalized to vanish at $0$. Since $\bdel (P+iQ)=0$, 
\begin{equation*}
	\oo{D}_{k,\phi}=\bdel+\tfrac{k}{2}\bdel \phi\wedge=e^{ik Q}\left(\bdel+\tfrac{k}{2}\bdel(\phi-P)\wedge\right)e^{-ik Q}=e^{ik Q}\oo{D}_{k,\phi-P}e^{-ik Q}.
\end{equation*}
With the $L^2$ inner product induced by $g$ on $\Gamma(B(0,2), \Lambda^{0,q})$, the formal adjoint operators are related by
\begin{equation*}
\oo{D}^*_{k,\phi}=e^{ik Q}\oo{D}^*_{k,\phi-P}e^{-ik Q}.
\end{equation*}
For the Laplace operators, we have
\begin{equation*}
	\Delta_{k,\phi}=e^{ik Q}\Delta_{k,\phi-P}e^{-ik Q}.
\end{equation*}
Therefore,
\begin{equation*}
	\Delta_{k,\phi}u=0 \quad \mbox{ if and only if } \quad \Delta_{k,\phi-P}e^{-ik P}u=0.
\end{equation*}   
What is more, $u$ and $e^{-ik P}u$ share the same $L^{\infty}$ norm. So by replacing $\phi$ by $\phi-P$, we can assume the Taylor expansion of $\phi$ at $0$ has the form
\begin{equation*}
	\phi(\zeta)=\sum_{i,j=1}^{n}\phi_{i\bar{j}}\zeta_i\oo{\zeta}_j+O(|\zeta|^3).
\end{equation*} 

We will now extend $\phi$ to a plurisubharmonic function on $\mathbb{C}^n$. Let $\eta$ be a smooth cut-off function on $\mathbb{C}^n$  such that $\eta=1$ on $B(0,1)$ and $\supp \eta\subseteq B(0,2)$. For any $\zeta\in \mathbb{C}^n$, define 
\begin{equation*}
\psi(\zeta)=\phi_2(\zeta)+\eta\left(\tfrac{\zeta}{\varepsilon_0}\right)\left(\phi(\zeta)-\phi_2(\zeta)\right),
\end{equation*}
where  $\varepsilon_0$ is a small positive number and $\phi_2$ is the Taylor polynomial of degree $2$ for $\phi$ at $0$, which in here $\phi_2(\zeta)=\sum_{i,j=1}^n\phi_{i\bar{j}}(0)\zeta_i\oo{\zeta}_j$. Then $\psi$ is close to $\phi_2$ in $C^2$ norm when $\varepsilon_0$ is sufficiently small. 
\begin{lemma}
	When $\varepsilon_0$ is sufficiently small, we have
	\begin{equation*}
	\|\psi-\phi_2\|_{C^2(\mathbb{C}^n)}=O(\varepsilon_0).
	\end{equation*}
\end{lemma} 

\begin{proof}
	Since $\psi-\phi_2= 0$ when $|\zeta|\geq 2\varepsilon_0$, we only need to consider $|\zeta|\leq 2\varepsilon_0$. Since $\varepsilon_0$ is sufficiently small, by Taylor's theorem, 
	\begin{equation*}
	\phi(\zeta)-\phi_2(\zeta)=O(|\zeta|^3).
	\end{equation*}
	Therefore, 
	\begin{equation*}
	\|\psi-\phi_2\|_{C^2(\mathbb{C}^n)}=\|\psi-\phi_2\|_{C^2(B(0,2\varepsilon_0))}=O(\varepsilon_0).
	\end{equation*}
\end{proof}

In particular, this implies that $\left(\partial_i\partial_{\bar{j}}\psi\right)(\zeta)$ is positive uniformly for $\zeta\in \mathbb{C}^n$ when $\varepsilon_0$ is sufficiently small, because $\partial_{i}\partial_{\bar{j}} \phi_2(z)=\partial_{i}\partial_{\bar{j}} \phi(0)$ is a positive definite constant matrix. From now on, we will fix such a sufficiently small constant $\eps_0$.

To proceed, we need the following lemma on the function $J(x)=M(x)^{1/x}$ on $\mathbb{R}_{>0}$. 
\begin{lemma}\label{g properties 1}
	If $\log M(x)$ is strictly convex and $J(x)$ is unbounded, then
	there exists some $x_0>0$ such that $J(x)$ is strictly increasing for $x>x_0$
\end{lemma}

\begin{proof}
	Since $\log M(x)$ is strictly convex, $x\log J(x)$ is strictly convex. By taking two derivatives, we have
	\begin{equation}\label{log convex}
	2\,\frac{J'(x)}{J(x)}+x\frac{J''(x)J(x)-(J'(x))^2}{J^2(x)}>0.
	\end{equation} 
	It follows that
	\begin{equation*}
	x\left(\frac{J'(x)}{J(x)}\right)'=x\frac{J''(x)J(x)-(J'(x))^2}{J^2(x)}
	>-2\frac{J'(x)}{J(x)}.
	\end{equation*}
	Therefore,
	\begin{equation}\label{alpha increasing}
	\left(x^2\frac{J'(x)}{J(x)}\right)'>0.
	\end{equation}
	So either $J'(x)$ is always negative, or $J'(x)$ is positive starting from some $x_0$. In the first case, $J'(x)<0$ implies that $J(x)$ is bounded, which contradicts our assumption.
	So $J'(x)$ is positive after some $x_0$, and the result follows immediately.
\end{proof}
\begin{defn} \label{x0} We define $x_0$ to be the smallest number such that $J'(x)>0$ for $x>x_0$ and $J(x) \geq J (y)$ for all $ y\leq x_0$. 
\end{defn}

Next, we construct a sequence of real analytic functions ${\Psi}_{N, r}$ to approximate $\psi$. Define
\begin{equation}\label{Definition PsiN}
{\Psi}_{N, r}(\zeta)=r^{-2}\psi_N(r\zeta)+r^{-2}\left(1-\eta(\zeta)\right)\left(\psi_2(r\zeta)-\psi_N(r\zeta)\right),
\end{equation}
where $\psi_N$ is the Taylor polynomial of degree $N$ for $\psi$ at $0$. We will restrict $r$ to be $$\frac{1}{\sqrt{k}}\leq r\leq \eps M(N)^{-1/N},$$
 for some variable $\eps>0$ sufficiently small, where all approximation functions ${\Psi}_{N, r}$ are uniformly plurisubharmonic on $\mathbb{C}^n$.

\begin{lemma}\label{PsiN Uniformly PSH}
	There exits some $\eps+1>0$ depending on constant $A$, dimension $n$ and the metric $g$, such that $\sqrt{-1}\del\bdel{\Psi}_{N, r}(\zeta)$ is positive uniformly for $\zeta\in \mathbb{C}^n$ and $r\leq \eps M(N)^{-1/N}$ for all $0<\eps\leq \eps_1$. 
\end{lemma}
\begin{proof}
	It is sufficient to show that when $\eps$ is small enough, uniformly for $r\leq \eps M(N)^{-1/N}$, one has
	\begin{equation} \label{PsiN C2}
	\|{\Psi}_{N, r}(\zeta)-\psi_2(\zeta)\|_{C^2(\mathbb{C}^n)}=O(\eps).
	\end{equation}
	By \eqref{Definition PsiN}, $\psi_{N,r}$ writes into
	\begin{equation*}
	{\Psi}_{N, r}(\zeta)=\psi_2(\zeta)+r^{-2}\eta(\zeta)\left(\psi_N(r\zeta)-\psi_2(r\zeta)\right).
	\end{equation*} 
	Since $\supp\eta\subseteq B(0,2)$,
	we only need to estimate $\|\frac{1}{r^2}\left(\psi_N(r\zeta)-\psi_2(r\zeta)\right)\|_{C^2(|\zeta|\leq 2)}$.
	Recall $\psi_N$ and $\psi_2$ are Taylor polynomials for $\psi$ at $0$. Denote $v=(\zeta,\oo{\zeta})$. Then
	\begin{align*}
	\tfrac{1}{r^2}\left(\psi_N(r\zeta)-\psi_2(r\zeta)\right)
	=\sum_{3\leq |\alpha|\leq N} \tfrac{D^{\alpha}_v\psi}{\alpha!}(0) r^{|\alpha|-2}v^{\alpha}.
	\end{align*} 
	By Lemma \ref{g properties 1} and Definition \ref{x0}, if $N \geq x_0$ then we have $M(|\alpha|)^{1/|\alpha|}\leq M(N)^{1/N}$ for any $1\leq |\alpha|\leq N$. Note that $\psi=\phi$ in $B(0,\eps_0)$ and thus we have
	\begin{align*}
	\left\|\tfrac{1}{r^2}\left(\psi_N(r\zeta)-\psi_2(r\zeta)\right)\right\|_{C^2(|\zeta|\leq 2)}
	\leq& \sum_{3\leq |\alpha|\leq N} \left|\tfrac{D^{\alpha}_v\psi}{\alpha!}(0)\right| r^{|\alpha|-2}\|v^{\alpha}\|_{C^2(|\zeta|\leq 2)}
	\\
	\leq & \sum_{3\leq |\alpha|\leq N} A^{|\alpha|-2} M(|\alpha|-2)\left({\eps } M(N)^{-1/N}\right)^{|\alpha|-2}|\alpha|^22^{|\alpha|}
	\\
	\leq &\sum_{3\leq |\alpha|\leq N} (\eps A)^{|\alpha|-2}|\alpha|^22^{|\alpha|}.
	\end{align*}
	Therefore, when $\eps \leq \frac{1}{16 A}$, 
	\begin{align*}
	\left\|\tfrac{1}{r^2}\left(\psi_N(r\zeta)-\psi_2(r\zeta)\right)\right\|_{C^2(|\zeta|\leq 2)}
	\leq  2^{2n+7}A\eps.
	\end{align*}
	So the result follows by the fact that $\tfrac{1}{r^2}\psi(r\zeta)=\sum_{i,j=1}^{n}\phi_{i\bar{j}}(0)\zeta_i\oo{\zeta}_{j}$.
\end{proof}

\begin{rmk}\label{PsiN Ck}
	By using the same argument, we can actually generalize \eqref{PsiN C2} to any $C^m$ norm. That is to say, 
	\begin{equation*}
	\|{\Psi}_{N, r}(\zeta)-\psi_2(\zeta)\|_{C^m(\mathbb{C}^n)}=O_m(\eps).
	\end{equation*}
\end{rmk}

The approximation functions ${\Psi}_{N, r}$ are also uniformly real analytic as stated in the following lemma.
\begin{lemma}\label{PsiN Uniformly Analytic}
	There exist some $\eps_1>0$ depending on $A$, such that ${\Psi}_{N, r}(\zeta)$ are real analytic on $|\zeta|\leq 1$, uniformly for $r\leq {\eps M(N)^{-1/N}}$ for all $0<\eps\leq \eps_1$. 
\end{lemma}

\begin{proof}
	We denote $v=(\zeta,\oo{\zeta})$. When $|\zeta|\leq 1$,  \eqref{Definition PsiN} writes into
	\begin{equation*}
	{\Psi}_{N, r}(\zeta)=r^{-2}\psi_N(r\zeta)=\sum_{2\leq |\alpha|\leq N}\tfrac{D_v^\alpha\psi}{\alpha!}(0)r^{|\alpha|-2} v^{\alpha}.
	\end{equation*}
	After taking derivatives, 
	\begin{equation*}
	\tfrac{D_v^{\beta}{\Psi}_{N, r}(\zeta)}{\beta!}=\sum_{2\leq |\alpha|\leq N}\sum_{\beta\leq \alpha}\tfrac{D_v^\alpha\psi}{\alpha!}(0)r^{|\alpha|-2} \binom{\alpha}{\beta} v^{\alpha-\beta}.
	\end{equation*}
	By using the fact that $\psi=\phi $ on $B(0,\eps_0)$, 
	\begin{align*}
	\left|\tfrac{D_v^{\beta}{\Psi}_{N, r}(\zeta)}{\beta!}\right|
	\leq & \sum_{2\leq |\alpha|\leq N}\sum_{\beta\leq \alpha}A^{|\alpha|-2}{M(|\alpha|-2)} r^{|\alpha|-2} \binom{\alpha}{\beta} 
	\\
	\leq & \sum_{2\leq |\alpha|\leq N}2^{|\alpha|} (A\eps)^{|\alpha|-2}. 
	\end{align*}
	When $\eps\leq \frac{1}{4A}$, 
	\begin{equation*}
	\left|\tfrac{D_v^{\beta}{\Psi}_{N, r}(\zeta)}{\beta!}\right|\leq 2^{2n+5}.
	\end{equation*}
\end{proof}

Similarly, we will also approximate the metric $g$ by analytic ones. Define 
\begin{equation}\label{gN}
	g_{N, r}(\zeta)=g_N(r\zeta)+\left(1-\eta(\zeta)\right)\left(g(0)-g_N(r\zeta)\right),
\end{equation}
where $g_N$ is the Taylor polynomial of degree $N$ for $g$ at $0$. By the same argument as in Lemma \ref{PsiN Uniformly PSH}, we can show that
\begin{equation}\label{gN Uniformly bounded}
\|g_{N, r}(\zeta)-g(0)\|_{C^2(\mathbb{C}^n)}=O(\eps).
\end{equation}
In particular, when $\eps$ is sufficiently small, $g_{N, r}$ is bounded below by some positive constant on $\mathbb{C}^n$, uniformly for $r\leq \eps M(N)^{-1/N}$. What is more, the same argument as in Lemma \ref{PsiN Uniformly Analytic} implies that $g_{N, r}$ is real analytic on $|\zeta|\leq 1$, uniformly for $r\leq \eps M(N)^{-1/N}$.

For each plurisubharmonic function ${\Psi}_{N, r}$ on $\mathbb{C}^n$ and any $k>0$, we can define $\oo{D}_{k,{\Psi}_{N-n-1,r}}=\bdel+\tfrac{k}{2}\bdel{\Psi}_{N-n-1,r}\wedge$. With the $L^2$ inner product on differential forms induced by $g_{N-n-3,r}$, we have the formal adjoint $\oo{D}_{k,{\Psi}_{N-n-1,r}}^*$. We can further define the Laplace operator $$\Delta_{k,{\Psi}_{N-n-1,r}, g_{N-n-3,r}}=\oo{D}_{k,{\Psi}_{N-n-1,r}}\oo{D}_{k,{\Psi}_{N-n-1,r}}^*+\oo{D}_{k,{\Psi}_{N-n-1,r}}^*\oo{D}_{k,{\Psi}_{N-n-1,r}}.$$ One key observation is that 
\begin{equation*}
	\oo{D}_{k r^2,{\Psi}_{N-n-1}}=\oo{D}_{k,r^2{\Psi}_{N-n-1}},
\end{equation*}
and so is that
\begin{equation*}
	\Delta_{k r^2,{\Psi}_{N-n-1,r},g_{N-n-3,r}}=\Delta_{k,r^2{\Psi}_{N-n-1,r},g_{N-n-3,r}}.
\end{equation*}

Denote $\Delta_N=\Delta_{k r^2,{\Psi}_{N-n-1,r},g_{N-n-3,r}}$ for simplicity and let $\Delta=\Delta_{k,\psi(r\zeta),g(r\zeta)}$ for $|\zeta|\leq 1$. We will compare these two Laplacians. 
\begin{lemma}
	For any $w\in C^{n+2}(\overline{B(0,1)})$ and any $|\zeta|\leq 1$, we have
	\begin{equation*}
	\left|\Delta w(\zeta)-\Delta_N w(\zeta)\right|_{C^{n}}\leq Ck^2 (2^{n+3}A\eps)^{N-n-2}|w(\zeta)|_{C^{n+2}},
	\end{equation*}  
	where $C$ is a constant only depending on the metric $g$ and dimension $n$ and $|w(\zeta)|_{C^m}= \sum_{j=0}^m|\partial^jw(\zeta)|$ for any $m\geq 0$.
\end{lemma}

\begin{proof}
Recall that when $|\zeta|\leq 1$, $r^2{\Psi}_{N, r}(\zeta)=\psi_N(r\zeta)$, which is the Taylor polynomial of $\psi$ at $0$. By Taylor's theorem,
\begin{align*}
	\psi(r\zeta)-\psi_{N-n-1}(r\zeta)=\sum_{|\alpha|=N-n}\tfrac{N-n}{\alpha!}(r\zeta)^{\alpha}\int_0^1(1-t)^{N-n-1} \left(D^{\alpha}\psi\right)(rt\zeta)dt.
\end{align*}
Allow $C(n)$ to be a constant only depending on $n$, which may change in different steps.
Thus, for $r\leq \eps M^{-1/N}$ and sufficiently small $\eps$, we can estimate the difference as
\begin{align*}
\left\|\psi(r\zeta)-\psi_{N-n-1}(r\zeta)\right\|_{C^{n+2}(|\zeta|\leq 1)}
\leq C(n)2^{(n+3)N} (\eps A)^{N-n}\leq C(n)(2^{n+3}\eps A)^{N-n},
\end{align*}
Similarly, $g_{N-n-3,r}(\zeta)=g_{N-n-3}(r\zeta)$ for $|\zeta|\leq 1$, which is the Taylor polynomial of $g$ at $0$, whence
\begin{align*}
\left\|g(r\zeta)-g_{N-n-3,r}(\zeta)\right\|_{C^{n+2}(|\zeta|\leq 1)}\leq C(n)(2^{n+3}A\eps)^{N-n-2}.
\end{align*}
Since $g_{N, r}$ are bounded below and above uniformly on $|\zeta|\leq 1$ by \eqref{gN Uniformly bounded}, we also have the similar estimates on the inverses of $g_{N, r}$ and $g$. That is,
\begin{align*}
\left\|g^{-1}(r\zeta)-g^{-1}_{N-n-3,r}(\zeta)\right\|_{C^{n+2}(|\zeta|\leq 1)}\leq C (2^{n+3}A\eps)^{N-n-2},
\end{align*}
where $C=C(n,g)$ is a constant depending only on the metric $g$ and dimension $n$.
Note that in general, the notation $\Delta_{k,\phi,g}$ denotes a second order differential operator, whose coefficients are polynomials of derivatives (up to second order) of $g, \phi$ and $g^{-1}$, together with $k$ and $k^2$. Thus, for any $w\in C^2(\overline{B(0,1)})$ and any $|\zeta|\leq 1$, we have
\begin{equation*}
\left|\Delta w(\zeta)-\Delta_N w(\zeta)\right|_{C^{n}}\leq Ck^2 (2^{n+3}A\eps)^{N-n-2}|w(\zeta)|_{C^{n+2}}.
\end{equation*}  
\end{proof}

Now define $w(\zeta)=u_r(\zeta):=u(r\zeta)$. Since for any $|\zeta|\leq 1$,
\begin{equation*}
	\Delta u_r(\zeta)=\Delta_{k,\psi(r\zeta),g(r\zeta)} u(r\zeta)=\Delta_{k,\phi(r\zeta),g(r\zeta)}u(r\zeta)=0,
\end{equation*}
by using the interior estimates \eqref{Interior Estimates}, we have for any $|\zeta|\leq \tfrac{1}{2}$,
\begin{equation}\label{Upper Bound of DeltaNu}
|\Delta_N u_r(\zeta)|_{C^{n}}\leq Ck^2 (2^{n+3}A\eps)^{N-n-2}|u_r(\zeta)|_{C^{n+2}}\leq Ck^{2n+5} (2^{n+3}A\eps)^{N-n-2}\|u_r\|_{L^2(B(\zeta,\tfrac{1}{2}))}.
\end{equation}

Recall that $\Delta_N=\Delta_{k r^2,{\Psi}_{N-n-1,r},g_{N-n-3,r}}$, where uniformly for $r\leq \eps M^{-1/N}$, $\sqrt{-1}\del\bdel {\Psi}_{N, r}$ is positive by Lemma \ref{PsiN Uniformly PSH} and $	\|g_{N, r}(\zeta)\|_{C^2(\mathbb{C}^n)}$ is bounded above by \eqref{gN Uniformly bounded}. Performing a standard integration by parts calculation (see Chapter IV in \cite{Ho}), for $\sqrt{\tfrac{\log k}{k}}\leq r\leq \eps  M^{-1/N}$ and any $h\in C^\infty_0(\mathbb{C}^n, \Lambda^{0,1})$,
\begin{equation*}
\left(\Delta_N h, h\right)_{L^2}
\geq \frac{k r^2}{2} \left(h, h\right)_{L^2}.
\end{equation*}
Note that the $L^2$ norm here is induced by the metric $g_{N-n-3}$ on $\mathbb{C}^n$. 

Let $\eta$ be a smooth cut-off function such that $\eta=1$ in $B(0,\tfrac{1}{4})$ and $\supp \eta\subseteq B(0,\tfrac{1}{2})$. Then for the equation
\begin{equation*}
\Delta_N v=\eta\Delta_Nu_r,
\end{equation*}
there exits a solution $v$ such that
\begin{equation}\label{L2 estimates}
\|v\|_{L^2}\leq \frac{2}{k r^2}\|\eta\Delta_Nu_r\|_{L^2}.
\end{equation}
Combining with \eqref{Upper Bound of DeltaNu}, we have
\begin{equation}\label{L2 Bound of v}
\|v\|_{L^2}\leq Ck^{2n+5}(2^{n+3}A\eps)^{N-n-2}\|u_r\|_{L^2(B(0,1))}.
\end{equation}
By the interior estimates \eqref{Interior Estimates} and the Sobolev embedding, we have the pointwise estimates of $v$ in $B(0,\tfrac{1}{4})$. That is,
\begin{align*}
	|v(\zeta)|
	\leq C\|v\|_{H^{n+1}(B(0,\frac{1}{2}))}
	\leq C\left(k^{n+1}\|v\|_{L^2(B(0,\frac{1}{2}))}+k^{n-1}\|\Delta_Nv\|_{H^{n}(B(0,\frac{1}{2}))}\right).
\end{align*} 
Using \eqref{Upper Bound of DeltaNu} and \eqref{L2 estimates}, we have
\begin{align*}
|v(\zeta)|
\leq Ck^{n+1}\|\Delta_N u_r\|_{H^n(B(0,\frac{1}{2}))}
\leq
Ck^{3n+6} (2^{n+3}A\eps)^{N-n-2}\|u_r\|_{L^2(B(\zeta,\tfrac{1}{2}))}.
\end{align*} 
On the other hand, as $\Delta_N(u_r-v)=0$ for $\zeta\in B(0,\tfrac{1}{4})$ and all the data in Remark \ref{uniform hypoanalyticity} are uniform for $\sqrt{\tfrac{\log k}{{k}}}\leq r\leq \eps M^{-1/N}$, by Lemma \ref{lemma hypoanalyticity}, 
\begin{equation*}
\|u_r-v\|_{L^\infty(B(0,\frac{1}{8}))}
\leq Ce^{-bk r^2} \|u_r-v\|_{L^\infty(B(0,\frac{1}{4}))}
\leq Ce^{-bk r^2}\left(1+k^{3n+6}(2^{n+3}A\eps)^{N-n-2}\right)\|u_r\|_{L^\infty(B(0,1))},
\end{equation*}
where $b$ is some positive constant independent from $k, r$ and $N$.
Therefore, 
\begin{align*}
\|u_r\|_{L^\infty(B(0,\frac{1}{8}))}
\leq Ck^{3n+6} e^{-bk r^2}\|u_r\|_{L^\infty(B(0,1))} +Ck^{3n+6}e^{-(N-n-2)\log\left(\frac{1}{2^{n+3}A\eps}\right)}\|u_r\|_{L^\infty(B(0,1))}.
\end{align*}
We restrict $N\geq 2n+4$ and $\eps\leq 2^{2n+6}A^2$ and rename $b$ to be $\min\left\{b,\frac{1}{4}\right\}$. Then,
\begin{align}\label{u bound 1}
\|u_r\|_{L^\infty(B(0,\frac{1}{8}))}
\leq Ck^{3n+6} \left(e^{-bk r^2} +e^{-bN\log\left(\frac{1}{\eps}\right)}\right)\|u_r\|_{L^\infty(B(0,1))},
\end{align}
where $C$ is a positive constant depending on the constant $A$ in \eqref{eq majorant} for $\partial^2\phi$ and $g$, the dimension $n$, the positivity of $\sqrt{-1}\partial\bar{\partial}\phi$ and $g$.

We now set $r=\eps M(N)^{-1/N}$ with $\eps\leq \eps_1$, where $\eps_1$ is sufficiently small so that \eqref{u bound 1} holds, and match the decay rates to obtain
\begin{equation*}
	k r^2=N\log\left(\tfrac{1}{\eps}\right).
\end{equation*}
To get the fastest decay rate, our goal is to maximize $k r^2$ under the above restrictions. 
\begin{lemma}\label{fatest decay}
	Given $k>x_0^2J(x_0)J'(x_0)e^{\frac{2NJ'(x_0)}{J(x_0)}}$, if $r\in (0,\infty)$, $\eps\in (0,\eps_1]$, $N\in [x_0,\infty)$ satisfy
	\begin{equation*}
		r=\eps M(N)^{-1/N}, \quad kr^2=N\log\left(1/\eps\right),
	\end{equation*}
	Then there exists a unique point $(\bar{r},\bar{\eps},\bar{N})$ maximizing $kr^2$ under the above conditions. What is more, the maximum satisfies
	\begin{equation}\label{decay rate comparable}
		C(\eps_1)^{-1}f^2(k)\log k\leq \max (kr^2) \leq f^2(k)\log k,
	\end{equation}
	where $f(k)$ is defined as in \eqref{growth f} and $C(\eps_1)$ is a constant depending on $\eps_1$ (for example we can take $C(\eps_1)=\frac{4\log(1/\eps_1)}{\eps_1^2}$).
\end{lemma}

Now we use this lemma to finish the proof of Lemma \ref{improved estimates} and the proof of Lemma \ref{fatest decay} is delayed to the next section. 

Since $(\bar{r},\bar{\eps},\bar{N})$ is the point where $k r^2$ is attaining the maximum, \eqref{decay rate comparable} implies
\begin{equation*}
	C(\eps_1)^{-1/2}f(k)\sqrt{\tfrac{\log k}{k}}\leq \bar{r}\leq f(k)\sqrt{\tfrac{\log k}{k}}. 
\end{equation*} 
Plugging $r=\delta \bar{r}$ into \eqref{u bound 1} for given $\delta\in (0,1)$ and using the fact that $\lim_{k\rightarrow \infty}f(k)=+\infty$ by Lemma \ref{f infinity}, we have
\begin{align*}
\|u\|_{L^\infty(B(0,\frac{\delta}{8C(\eps_1)^{1/2}}f(k)\sqrt{\frac{\log k}{k}})}
&\leq Ck^{3n+6}\left(e^{-b\delta^2\max (kr^2)}+e^{-b\max (kr^2)}\right) \|u\|_{L^\infty(B(0,\delta f(k)\sqrt{\frac{\log k}{k}}))}
\\
&\leq C e^{-\frac12 b\delta^2 f^2(k)\log k}\|u\|_{L^\infty(B(0,\delta f(k)\sqrt{\frac{\log k}{k}}))}.
\end{align*}
Here $C$ is a positive constant depending on the constant $A$ in \eqref{eq majorant} for $\partial^2\phi$ and $g$, the dimension $n$, the positivity of $\sqrt{-1}\partial\bar{\partial}\phi$ and $g$. We then rename $\frac12 b\delta^2$ by $b$. The result follows by a covering argument.

\subsection{The optimization problem for $f(k)$}\label{f(k)}
In this section we prove Lemmas \ref{fatest decay} and \ref{f infinity}. 

Note 
\begin{equation*}
k r^2=k \eps^2 M^{-2/N}=N\log\left(\tfrac{1}{\eps}\right).
\end{equation*} 
This can be written as
\begin{equation}\label{constraint}
	\frac{k\eps^2}{\log(\tfrac{1}{\eps})}=N M^{2/N}:=g(N).
\end{equation}
By Lemma \ref{g properties 1}, $g(N)=NM^{2/N}=NJ^2(N)$ is strictly increasing for $N\geq x_0$. We set $N\geq x_0$ from now on. Equivalently, that means 
\begin{equation*}
	\frac{k\eps^2}{\log\left(\frac{1}{\eps}\right)}=g(N)\geq g(x_0). 
\end{equation*}
Let $\eps_0(k)$ be the only solution to $\frac{k\eps^2}{\log\left(\frac{1}{\eps}\right)}=g(x_0)$. Then given $\eps\geq \eps_0(k)$ we can solve $N=N(\eps)=g^{-1}\left(\frac{k\eps^2}{\log\left(\frac{1}{\eps}\right)}\right)$. For $\eps\in [\eps_0(k),1)$, we set
\begin{equation*}
h(\eps)=N(\eps)\log\left(\frac{1}{\eps}\right).
\end{equation*}

Then $h$ satisfies the following property.
\begin{lemma}\label{property h}
	Given $k>0$, $h(\eps)$ has a unique critical point in $(\eps_0(k),1)$, denoted by $\eps(k)$. $h$ is strictly increasing in $(0, \eps(k))$ and strictly decreasing in $(\eps(k),1)$.
\end{lemma}

\begin{proof}
	We compute the critical point of $h$, 
	\begin{equation*}
	h'(\eps)=N'(\eps)\log\left(\frac{1}{\eps}\right)-\frac{N}{\eps}.
	\end{equation*} 
	On the other hand, if we rewrite \eqref{constraint} into
	\begin{equation*}
	k\eps^2=g(N)\log\left(\frac{1}{\eps}\right),
	\end{equation*}
	and take the derivative, we get
	\begin{equation*}
	2k\eps =g'(N)N'(\eps)\log\left(\frac{1}{\eps}\right)-\frac{g(N)}{\eps}.
	\end{equation*}
	Therefore,
	\begin{equation}\label{derivative of N}
	N'(\eps)\log\left(\frac{1}{\eps}\right)=\frac{2k\eps +\frac{g(N)}{\eps}}{g'(N)},
	\end{equation}
	and $h'(\eps)$ turns into
	\begin{equation}\label{first derivative}
	h'(\eps)=\frac{2k\eps +\frac{M^2}{\eps}}{g'(N)}-\frac{g(N)}{\eps}
	=\frac{2\log\left(\frac{1}{\eps}\right)g(N)+g(N)-Ng'(N)}{\eps g'(N)}.
	\end{equation}
	If we set $h'(\eps)=0$, then the critical point satisfies
	\begin{equation*}
	\log\left(\frac{1}{\eps}\right)=\frac{Ng'(N)-g(N)}{2g(N)}=\frac{NJ'(N)}{J(N)}.
	\end{equation*}
	Plugging this back into \eqref{constraint}, 
	\begin{equation} \label{eq to solve N}
	k=g\cdot \frac{Ng'(N)-g(N)}{2g(N)} \cdot e^{\frac{Ng'(N)-g(N)}{g(N)}}=N^2J(N)J'(N)e^{\frac{2NJ'}{J}}.
	\end{equation}
We will show that the function on the right side is strictly increasing to infinity when $N>x_0$.
\begin{lemma}\label{G is going to infinity}
	Let $G(x)=\frac{xg'(x)-g(x)}{g(x)}$ for $x\in \mathbb{R}_{>0}$ and let $x_0$ be defined by Definition \ref{x0}. If $\log M(x)$ is strictly convex and $M(x)^{\frac{1}{x}}$ is unbounded, then $g(x)\cdot G(x)\cdot e^{G(x)}$ is strictly increasing on $(x_0,\infty)$ and converges to $+\infty$ as $x\rightarrow \infty$.
	\end{lemma}
\begin{proof}
	Recall that $g(x)= x M^{\frac{2}{x}}(x)$ and thus $x\log(g(x))-x\log x$ is strictly convex. By a straightforward computation,
	\begin{equation}\label{g log convex}
	\left(x\log(g(x))-x\log x\right)''=2\frac{g'(x)}{g(x)}+x\frac{g''(x)g(x)-(g'(x))^2}{g^2(x)}-\frac{1}{x}>0.
	\end{equation}

	By Lemma \ref{g properties 1}, $\frac{g(x)}{x}$ is strictly increasing on $(x_0,\infty)$. It implies that for $x>x_0$,
	\begin{equation*}
	\left(\frac{g(x)}{x}\right)'=\frac{xg'(x)-g(x)}{x^2}>0,
	\end{equation*}
	and thus $G(x)>0$ for $x>x_0$.	
	Next we compute the derivative of $G$:
	\begin{equation*}
	G'(x)=\frac{g'(x)}{g(x)}+x\frac{g''(x)g(x)-\left(g'(x)\right)^2}{g^2(x)}.
	\end{equation*}
	So \eqref{g log convex} rewrites into 
	\begin{equation*}
	G'(x)-\frac{1}{x}G(x)=G'(x)+\frac{g'(x)}{g(x)}-\frac{1}{x}>0.
	\end{equation*}
	In order to show that $\ln g(x)+\ln G(x)+G(x)$ is increasing, we compute its derivative:
	\begin{align*}
	\left(\ln g(x)+\ln G(x)+G(x)\right)'
	=\frac{g'(x)}{g(x)}+\frac{G'(x)}{G(x)}+G'(x)>\frac{G'(x)}{G(x)}+\frac{1}{x}>0.
	\end{align*}
	Now that $g(x)\cdot G(x)\cdot e^{G(x)}$ is strictly increasing for $x>x_0$, it remains to check $g(x)\cdot G(x)\cdot e^{G(x)}$ is unbounded. Assume not. Then there exists some constant $C>0$ such that
	\begin{equation*}
	xg'(x)-g(x)=g(x)\cdot G(x)\leq C,
	\end{equation*}
	which can be rewritten into
	\begin{equation*}
	\left(\frac{g(x)}{x}+\frac{C}{x}\right)'\leq 0,
	\end{equation*}
	which is against the assumption that $\frac{g(x)}{x}=M(x)^{2/x}$ is unbounded.
	So the result follows.
\end{proof}
Now that the right side of \eqref{eq to solve N} is strictly increasing in $N$, for any $k>\frac{1}{2}g(x_0)G(x_0)e^{G(x_0)}$, there exists a unique solution $N(k)>x_0$ as in \eqref{solution N}. And by using the relation $\frac{k \eps^2}{\log \frac{1}{\eps}}=g(N)$ in \eqref{constraint}, we can further solve $\eps(k)\in(0,1)$, the critical point of $h(\eps)$. It is clear that $\eps(k)\in (\eps_0(k),1)$ by the fact $g(N(k))>g(x_0)$. 

Since $h(\eps)$ has a unique critical point in $(\eps_0(k),1)$, it remains to check $h$ attains a local max at $\eps(k)$. By taking one more derivative of \eqref{first derivative},
\begin{align*}
h''(\eps)
=&-\frac{2\log\left(\frac{1}{\eps}\right)+3}{\eps^2}\frac{g}{g'}+\frac{2\log\left(\frac{1}{\eps}\right)+1}{\eps}\frac{\left(g'\right)^2-gg''}{(g')^2}-\frac{N'(\eps)}{\eps}+\frac{N}{\eps^2}\\
=&-\frac{2\log\left(\frac{1}{\eps}\right)+3}{\eps^2}\frac{g}{g'}+\frac{2\log\left(\frac{1}{\eps}\right)+1}{\eps}\frac{\left(g'\right)^2-gg''}{(g')^2}-\frac{1}{\eps^2\log\left(\frac{1}{\eps}\right)}\frac{2g\log\left(\frac{1}{\eps}\right)+g}{g'}+\frac{N}{\eps^2}.
\end{align*} 
Using \eqref{log convex},
\begin{equation*}
h''(\eps)\leq -\frac{2\log\left(\frac{1}{\eps}\right)+3}{\eps^2}\frac{g}{g'}+\frac{2\log\left(\frac{1}{\eps}\right)+1}{\eps}\left(\frac{2g'}{g}-\frac{1}{N}\right)\left(\frac{g}{g'}\right)^2\frac{1}{N}-\frac{2\log\left(\frac{1}{\eps}\right)+1}{\eps^2\log\left(\frac{1}{\eps}\right)}\frac{g}{g'}+\frac{N}{\eps^2}.
\end{equation*}
Since at the critical point 
\begin{equation*}
\frac{g(\eps(k))}{g'(\eps(k))}=\frac{N}{2\log\left(\frac{1}{\eps(k)}\right)+1},
\end{equation*} 
it follows that
\begin{align*}
h''(\eps(k))\leq -\frac{2}{\eps(k)^2}\frac{N}{2\log\left(\frac{1}{\eps(k)}\right)+1}+\frac{2}{\eps(k)}-\frac{1}{\eps(k)\left(2\log\left(\frac{1}{\eps(k)}\right)+1\right)}-\frac{N}{\eps(k)^2\log\left(\frac{1}{\eps(k)}\right)}.
\end{align*}
Thus $h''(\eps(k))<0$ by noting the fact $$\max_{0<\eps<1}\eps\log\left(\frac{1}{\eps}\right)=\frac{1}{e}<\frac{1}{2}.$$
So the function $h$ is strictly increasing before the critical point $\eps(k)$ and is strictly decreasing afterwards. 
\end{proof}

As $\eps$ is actually contained in $(\eps_0(k),\eps_1]$ under our constraint, there are two cases for the maximum of $h(\eps)$. The first case is when $\eps(k)\leq \eps_1$. Then 
$$\max_{\eps\in (\eps_0(k),\eps_1]}h(\eps)
=N(k)\log\left(\frac{1}{\eps(k)}\right).$$
The second case is when $\eps(k)>\eps_1$ and   
the maximum of $h(\eps)$ is attained at $\eps_1$ instead. But in this case, $h(\eps_1)$ is actually comparable to $h(\eps(k))$. To be precise, we claim that there exists some constant $C(\eps_1)$ depending on $\eps_1$ such that
\begin{equation}\label{h comparable}
	h(\eps_1)\leq h(\eps(k))\leq C(\eps_1)h(\eps_1).
\end{equation}
The first inequality is clear by Lemma \ref{property h}. To prove the second one, we need the following lemma on $g$.
\begin{lemma}\label{g properties}
	If $\log M(x)$ is strictly convex and $M(x)^{\frac{1}{x}}$ is unbounded, then
	\begin{itemize}
		\item[(a)] $g(x)$ is strictly convex on $\mathbb{R}_{>0}$.
		\item[(b)] For any $y\geq 1$ and any $k\geq g(x_0)$,
				\begin{equation*}
				\frac{1}{y}g^{-1}(k y)\leq g^{-1}(2k). 
				\end{equation*}
	\end{itemize}
\end{lemma}
\begin{proof}
	Recall \eqref{g log convex} and rewrite it into
	\begin{equation*}
		\frac{g''(x)}{g(x)}>\left(\frac{g'(x)}{g(x)}\right)^2-\frac{2}{x}\,\frac{g'(x)}{g(x)}+\frac{1}{x^2}=\left(\frac{g'(x)}{g(x)}-\frac{1}{x}\right)^2.
	\end{equation*}
	Thus $(a)$ follows immediately.
	
	Now we prove $(b)$. Since $g(x)$ is increasing for $x\in (x_0,\infty)$ by Lemma \ref{g properties 1} and convex by $(a)$, the inverse function $g^{-1}(x)$ is concave on $(g(x_0),\infty)$. For any $k\geq g(x_0)$ and $y\geq 1$, we have
	\begin{equation*}
		\frac{1}{y}g^{-1}\left(k y\right)
		\leq \frac{1}{y}g^{-1}\left(k y\right)+\frac{y-1}{y}g^{-1}(g(x_0))
		\leq g^{-1}\left(\frac{yk +(y-1)g(x_0)}{y}\right)
		\leq g^{-1}(2k).
	\end{equation*}
	So we obtain $(b)$.
\end{proof}

Now we prove \eqref{h comparable}. If $\eps(k)\in (\eps_1,e^{-1})$, then
\begin{equation*}
h(\eps(k))=\log\left(1/\eps(k)\right)g^{-1}\left(\frac{k\eps(k)^2}{\log\left(1/\eps(k)\right)}\right)
\leq
\log\left(1/\eps_1\right) g^{-1}(k),
\end{equation*}
and Lemma \ref{g properties} gives that for any $k\geq \frac{2g(x_0)\log(1/\eps_1)}{\eps_1^2}$,
\begin{equation*}
	g^{-1}(k)\leq \frac{2\log(1/\eps_1)}{\eps_1^2}g^{-1}\left(\frac{k\eps^2_1}{\log\left(1/\eps_1\right)}\right).
\end{equation*}
Therefore, \eqref{h comparable} follows by setting $C(\eps_1)=\frac{2\log(1/\eps_1)}{\eps_1^2}$.

On the other hand, if $\eps(k)\in(e^{-1},1)$, then using Lemma \ref{g properties} again, for any $k\geq \frac{2\log(1/\eps_1)}{\eps_1^2}$, we have
\begin{equation*}
	h(\eps(k))
	\leq
	g^{-1}\left(2k\right)
	\leq \frac{4\log(1/\eps_1)}{\eps_1^2} g^{-1}\left(\frac{k\eps_1^2}{\log\left(1/\eps_1\right)}\right),
\end{equation*}
and thus \eqref{h comparable} follows by setting $C(\eps_1)=\frac{4}{\eps_1^2}$.

\begin{proof}[\textbf{Proof of Lemma \ref{f infinity} }]
	We denote $\beta(N(k))=\frac{J'(N(k))}{J(N(k))}$ and then $f^2(k)$ writes into
	\begin{equation*}
		f^2(k)=\frac{N^2(k) \beta(N(k))}{\log k}.
	\end{equation*}
	We denote $\dot{N}(k)=\frac{dN}{dk}(k)$ and compute the derivative of $f^2(k)$:
	\begin{equation} \label{f derivative}
		(f^2)'(k)=\frac{\dot{N}N\left(2\beta+N\beta'\right)k\log k-N^2\beta}{k\log^2k}.
	\end{equation}
	Taking the logarithm of \eqref{solution N}, we get
	\begin{equation}\label{log solution N}
		2\log N+ 2\log J+\log\beta+2N\beta=\log k.
	\end{equation}
	Differentiating this, we get
	\begin{equation*}
		\dot{N}\left(\frac{2}{N}+2\beta+\frac{\beta'}{\beta}+2\beta+2N\beta'\right)=\dot{N}\left(N\beta'+2\beta\right)\left(\frac{1}{N\beta}+2\right)=\frac{1}{k}.
	\end{equation*}
	We solve $\dot{N}$ and plug it into \eqref{f derivative}:
	\begin{equation*}
		(f^2)'(k)=\frac{N\left(\frac{1}{N\beta}+2\right)^{-1}\log k-N^2\beta}{k\log^2k}
				 =\frac{N\log k-N-2N^2\beta}{k\log^2k \left(\frac{1}{N\beta}+2\right)}.
	\end{equation*}
	Note that $\beta(N(k))=\frac{J'(N(k))}{J(N(k))}>0$ for $N(k)>x_0$ by Lemma \ref{g properties 1}. So it is sufficient to check $\log k-1-2N\beta$ is positive. By \eqref{log solution N},
	\begin{equation*}
		\log k-1-2N\beta=2\log J+\log\left(N^2\beta\right)-1.
	\end{equation*}
	By \eqref{alpha increasing}, $\log\left(N^2\beta\right)$ is bounded below when $N(k)>x_0$. Recall $J(x)$ is increasing to infinity when $x>x_0$ by Lemma 
	\ref{g properties 1} and $N(k)$ is also increasing to infinity. Thus $f(k)$ is strictly increasing when $k$ is large enough. It remains to show that $\lim_{k\rightarrow\infty}f(k)=\infty$. 
	
	Recall from the proof of Lemma \ref{property h} that
	\begin{equation*}
		f^2(k)\log k=h(\eps(k))\geq h(k^{-1/4}).
	\end{equation*}
	The second inequality is valid because $h$ attains the maximum at $\eps(k)$. If $\eps=k^{-1/4}$, then for sufficiently large $k$ satisfying $\frac{4\sqrt{k}}{\log k}\geq x_0$, $N(\eps)= g^{-1}\left(\frac{4\sqrt{k}}{\log k}\right)$ by \eqref{constraint}. Thus,
	\begin{equation*}
	h(k^{-1/4})= \frac{1}{4} g^{-1}\left(\frac{4\sqrt{k}}{\log k}\right) \log k.
	\end{equation*}
	The result follows by the fact that $g(x)=xJ^2(x)$ is increasing to infinity on $(x_0,\infty)$. 
\end{proof} 

\subsection{The second proof of Theorem \ref{Estimate of G in a shrinking neighborhood} in the analytic case.} \label{2ndanalytic} Let $s \in (1, 2)$ and assume $h \in G^1 \subset G^s$. Clearly the majorant is given by $M(N)= N^{(s-1)N}$ and hence correspondingly $J(N)= N^{s-1}$ and $g(N)=NJ^2(N)= N^{2s-1}$. It is clear from Definition \ref{x0} that in this case $x_0 =1$, and therefore $k_0=(s-1)e^{2s-2}$. From the proof of Lemma \ref{g properties}, we also require that $k\geq \frac{2\log(1/\eps_1)}{\eps_1^2}$, where $\eps_1$ only depends $n$ and $h$ but not $s$. In addition, in the proof of estimate \eqref{u bound 1}, we need $N(k) \geq 2n+4$. But since $N(k)$ satisfies \eqref{solution N} we get 
$$N(k) = \left ( \frac{e^{2-2s}}{s-1} \right )^{\frac{1}{2s-1}} k^{\frac{1}{2s-1}} \geq \left ( \frac{e^{2-2s}}{s-1} \right )^{\frac{1}{2s-1}}.$$
Hence by choosing $s$ sufficiently close to $1$ we get $N(k) \geq 2n+4$ for all $k \geq 1$. 
Another place that we need to carefully study the dependence of our constants on $s$ is in the proof  of Lemma \ref{fatest decay}, when we absorb $k^{a(n)}$ into $e^{-bf(k)^2 \log k}$, where $a(n)$ depends only on $n$.  To do this we need to choose $k$ large enough so that $b f(k)^2  \geq 2 a(n)$.   This follows because by \eqref{growth f} for $s$ sufficiently close to $1$, we have
$$ f(k) = \left ( \frac{(s-1)N(k)}{\log k}\right )^{\frac12}=\left ( s-1\right )^{\frac{s-1}{2s-1}} e^{\frac{1-s}{2s-1}} \frac{k^{\frac{1}{4s-2}}}{\sqrt{\log k}}\geq \frac{1}{2} \frac{k^{\frac{1}{6}}}{\sqrt{\log k}}. $$
Finally, since $b$ and $C$ in Theorem \ref{Estimate of G in a shrinking neighborhood} are independent of $s$, we can take the limit as $s\rightarrow 1^{+}$, and obtain $f(k)\rightarrow \frac{k^{\frac{1}{2}}}{\sqrt{\log k}}$, which proves Theorem \ref{Estimate of G in a shrinking neighborhood} in the analytic case.

\subsection{Estimates of the Bergman kernel when $d(z, w) \leq \gamma \sqrt{\frac{\log k}{k}}$}\label{verynear}

In the shrinking neighborhood $d(z,w)\leq \gamma \sqrt{\frac{\log k}{k}}$ for any $\gamma>0$, the asymptotic expansion of Bergman kernel \eqref{ZC} is actually still valid. The following theorem is due to Shiffman and Zelditch in \cite{ShZe} and \cite{ShZeGAFA}. We include a proof here for the completeness using \cite{BBS}.
\begin{thm}Assume $h\in C^{\infty}$.
	Given any positive constant $\gamma$, if $d(z,w)<\gamma \sqrt{\frac{\log k}{k}}$, then we have
	\begin{equation}\label{complete expansion}
	K_k(z,w)=e^{-k\psi(z,\bar{w})}\frac{k^n}{\pi^n}\left(1+\sum_{j=1}^{m-1}\frac{b_j(z,\bar{w})}{k^j}+\frac{1}{k^m}O_{m, \gamma}(1)\right).
	\end{equation}
	Here, $\psi(z, \bar w)$ and $b_j(z, \bar w)$ are (almost) holomorphic extensions of $\phi(z)$ and $b_j(z, \bar z)$ from \eqref{ZC}.
\end{thm}
	Note that by taking $m=1$, \eqref{complete expansion} writes into \eqref{Shrinking}, which in particular implies the desired estimates \eqref{eq christ} of the Bergman kernel for $d(z,w)\leq \gamma \sqrt{\frac{\log k}{k}}$.
	  
\begin{proof}
	For any $m\in \mathbb{N}$, when $d(z,w)$ is sufficiently small, we have the off-diagonal expansion from \cite{BBS}:
	\begin{equation*}
	K_k(z,w)=e^{k\psi(z,\bar{w})}\frac{k^n}{\pi^n}\left(1+\sum_{j=1}^{m-1}\frac{b_j(z,\bar{w})}{k^j}\right)+e^{\frac{k}{2}\left(\phi(z)+\phi(w)\right)}k^{-m+n}O_m(1).
	\end{equation*}
	By combining the terms on the right hand side together, we obtain
	\begin{equation*}
	K_k(z,w)=e^{k\psi(z,\bar{w})}\frac{k^n}{\pi^n}\left(1+\sum_{j=1}^{m-1}\frac{b_j(z,\bar{w})}{k^j}+e^{\frac{k}{2}\left(\phi(z)+\phi(w)-2\psi(z,\bar{w})\right)}k^{-m}O_m(1)\right).
	\end{equation*}
	Now we replace $m$ by $m+p$ in the above equation. Then
	\begin{align*}
	K_k(z,w)
	=e^{k\psi(z,\bar{w})}\frac{k^n}{\pi^n}\left(1+\sum_{j=1}^{m+p-1}\frac{b_j(z,\bar{w})}{k^j}+e^{\frac{k}{2}\left(\phi(z)+\phi(w)-2\psi(z,\bar{w})\right)}k^{-m-p}O_{m+p}(1)\right).
	\end{align*}
	For the error term, if we take $p\geq\gamma^2$ we have
	\begin{align*}
	\left|e^{\frac{k}{2}\left(\phi(z)+\phi(w)-2\psi(z,\bar{w})\right)}\right|=e^{\frac{k}{2}D(z,w)}\leq e^{kd^2(z,w)}
	\leq k^{\gamma^2}\leq k^p.
	\end{align*}
	Therefore,
	\begin{align*}
	K_k(z,w)
	=&e^{-k\psi(z,\bar{w})}\frac{k^n}{\pi^n}\left(1+\sum_{j=1}^{m-1}\frac{b_j(z,\bar{w})}{k^j}+\sum_{j=m}^{m+p-1}\frac{b_j(z,\bar{w})}{k^j}+k^p\cdot k^{-m-p}O_{m, \gamma}(1)\right)\\
	=&e^{-k\psi(z,\bar{w})}\frac{k^n}{\pi^n}\left(1+\sum_{j=1}^{m-1}\frac{b_j(z,\bar{w})}{k^j}+ k^{-m}O_{m, \gamma}(1)\right).
	\end{align*}
	So the result follows.
\end{proof}

\section{Estimates of the Green kernel far from the diagonal}\label{far}
In this section, we are going to prove Theorem \ref{main}. By the relation \eqref{Bergman and Green Kernel} between the Bergman kernel and Green kernel. it is sufficient to estimate $G_k(z,w)$ and its derivatives. 
\subsection{Construction of local Green kernel}
\begin{lemma}\label{Local Green Kernel}
	Given any $r$ with $r\geq f(k)\sqrt{\frac{\log k}{k}}$, there exists a bounded linear operator $$T: L^2(X,\Lambda^{0,1}(L^{k}))\rightarrow L^2(X,\Lambda^{0,1}(L^{k}))$$ such that for sufficiently large $k$,
	\begin{itemize}
		\item[(1)] The distribution kernel of $T$ is supported in $\{(z,w): d(z,w)\leq r\}$,
		\item[(2)] $\|T\circ \Box_{k}-I\|\leq e^{-brf(k)\sqrt{k\log k}}$, where $b$ is a positive constant depending on $(L,h)$ and $X$.
	\end{itemize}
\end{lemma}
\begin{proof}
	Fix $r$ with $r\geq f(k)\sqrt{\frac{\log k}{k}}$. For any $x,y\in X$, define $$K(z,w)=G_{k}(z,w)\eta\left(\frac{k}{f^2(k)\log k}d(z,w)^2\right),$$ where $\eta$ is a smooth cut-off function such that $\supp\eta\subseteq B(0,1)$ and $\eta=1$ on $B(0,\tfrac{1}{4})$. Then clearly we have $$\supp K\subseteq \{(z,w)\in X\times X :d(z,w)\leq f(k)\sqrt{\tfrac{\log k}{k}}\}.$$ Define $$P:L^2(X,\Lambda^{0,1}(L^{k}))\rightarrow L^2(X,\Lambda^{0,1}(L^{k}))$$ as the operator with distribution kernel $K(z,w)$. Letting $\Box_{k}$ act on $z$ variable and computing derivatives straightforwardly, 
	\begin{equation*}
	\left|\Box_{k}\left(K(z,w)-G_{k}(z,w)\right)\right|
	\leq Ck\left(|G_{k}(z,w)|+|\nabla_z G_{k}(z,w)|\right).
	\end{equation*}
	Since $\left|\Box_{k}\left(K(z,w)-G_{k}(z,w)\right)\right|$ is supported in $$\{(z,w)\in X\times X :\tfrac{1}{2}f(k)\sqrt{\tfrac{\log k}{k}}\leq d(z,w)\leq f(k)\sqrt{\tfrac{\log k}{k}}\},$$ so by Theorem \ref{Estimate of G in a shrinking neighborhood}, we have
	\begin{equation*}
		|G_{k}(z,w)|+|\nabla_x G_{k}(z,w)|\leq C e^{-b f^2(k)\log k},
	\end{equation*}
	and therefore we must have
	\begin{equation*}
	\left|\Box_{k}\left(K(z,w)-G_{k}(z,w)\right)\right|\leq C{k}^{2} e^{-bf^2(k)\log k}.
	\end{equation*}
	In terms of the operator $P$, the above inequality becomes
	\begin{equation*}
	\left\|\Box_{k}\circ P-I\right\|\leq C{k}^{2} e^{-b f^2(k)\log k}.
	\end{equation*}
	If we change $b$ to $\tfrac{b}{2}$, then for sufficiently large $k$,
	\begin{equation*}
	\left\|\Box_{k}\circ P-I\right\|\leq  e^{-bf^2(k)\log k}.
	\end{equation*} 
	Next define operators 
	\begin{align*}
	E=I-\Box_{k}\circ P, \qquad T=P\circ \sum_{j=0}^{N-1}E^j.
	\end{align*}
	Then 
	\begin{equation*}
	\Box_{k}\circ T=(I-E)\circ \sum_{j=0}^{N-1}E^j=I-E^N.
	\end{equation*}
	Taking the operator norm, we get
	\begin{equation*}
	\left\|\Box_{k}\circ T-I\right\|\leq \|E\|^N\leq e^{-bNf^2(k)\log k}.
	\end{equation*}
	Set $N=[\frac{\sqrt{k}r}{f(k)\sqrt{\log k}}]\geq \tfrac{1}{2}\frac{\sqrt{k}}{f(k)\sqrt{\log k}}r$. Then  
	\begin{equation*}
	\|\Box_{k}\circ T-I\|\leq e^{-\tfrac{b\,r}{2}f(k)\sqrt{k\log k}}.
	\end{equation*}
	Since $\Box_{k}$ and $T$ are both formally self-adjoint, we immediately get the same bound for $T\circ \Box_{k}-I$.
	Recall $\supp K\subseteq \{(z,w)\in X\times X: d(z,w)\leq f(k)\sqrt{\frac{\log k}{k}}\}$, so is the distribution kernel of $E$. Therefore, the distribution kernel of $T$ is supported where
	\begin{equation*}
	d(z,w)\leq Nf(k)\sqrt{\tfrac{\log k}{k}}\leq r.
	\end{equation*}
\end{proof}

Now we will prove Theorem \ref{main}. First, note that when $d(z,w)\leq \gamma \sqrt{\frac{\log k}{k}}$, the asymptotic expansion \eqref{Shrinking} is valid and \eqref{eq christ} directly follows, and when $\gamma \sqrt{\frac{\log k}{k}} \leq d(z,w) \leq f(k)\sqrt{\frac{\log k}{k}}$, \eqref{eq christ} follows by Theorem \ref{Estimate of G in a shrinking neighborhood}. It remains to prove the result for $d(z,w)\geq f(k)\sqrt{\frac{\log k}{k}}$.

\subsection{Estimates of the Green kernel when $d(z,w)\geq f(k)\sqrt{\frac{\log k}{k}}$}

\begin{lemma}\label{Green Kernel Estimate for nearby points}
	There exist positive constants $C$ and $b$ such that for any $z,w\in X$ with $d(z,w)\geq f(k)\sqrt{\frac{\log k}{k}}$, we have
	\begin{equation*}
	|{G}_{k}(z,w)|_{C^2}\leq C e^{-bf(k)\sqrt{k\log k}\,d(z,w)}.
	\end{equation*}
\end{lemma}

\begin{proof}
	Given $z,w\in X$ with $d(z,w)\geq f(k)\sqrt{\frac{\log k}{k}}$. By taking necessary embedding and restriction, we can regard $G_{k}$ as an operator from $L^2(B(y,\frac{1}{4}d(z,w)),\Lambda^{0,1}(L^{k}))$ to $L^2(B(x,\frac{1}{4}d(z,w)),\Lambda^{0,1}(L^{k}))$.  Set $r=d(z,w)$. Let $T(z',w')$ be a local Green kernel constructed in Lemma \ref{Local Green Kernel} supported where $d(z',w')\leq \frac{r}{2}$ and $T$ be the associated operator. Then for any $u\in L^2(B(w,\frac{1}{4}d(z,w)),\Lambda^{0,1}(L^{k}))$ with $\supp u\subseteq B(w,\frac{1}{4}d(z,w))$, we have
	\begin{align*}
	G_{k}u=T\Box_{k}G_{k}u+\left(I-T\Box_{k}\right)G_{k}u
	=Tu+\left(I-T\Box_{k}\right)G_{k}u.
	\end{align*} 
	Since $\supp Tu\subseteq B(w,\frac{1}{4}d(z,w)+\frac{r}{2})$, which is disjoint from $B(z,\frac{1}{4}d(z,w))$,
	\begin{align*}
	\|G_{k}u\|_{L^2(B(z,\frac{1}{4}d(z,w)))}
	=&\|\left(I-T\Box_{k}\right)G_{k}u\|_{L^2(B(z,\frac{1}{4}d(z,w)))}\\
	\leq& \tfrac{2}{k}e^{-b r f(k)\sqrt{k\log k}} \|u\|_{L^2(B(w,\frac{1}{4}d(z,w)))}.
	\end{align*}
	The second inequality follows by the second part of Lemma \ref{Local Green Kernel} and the fact that $\Box_{k}$ is bounded below by $\tfrac{k}{2}$ for sufficiently large $k$.
	Thus as an operator from $L^2(B(w,\frac{1}{4}d(z,w)),\Lambda^{0,1}(L^{k}))\rightarrow L^2(B(z,\frac{1}{4}d(z,w)),\Lambda^{0,1}(L^{k}))$, 
	\begin{equation*}
		\|G_{k}\|\leq \tfrac{2}{k}e^{-b r f(k)\sqrt{k\log k}}\leq e^{-b r f(k)\sqrt{k\log k}}.
	\end{equation*}
	Since $\Box_{k}G_{k}u=u$, which vanishes in $B(z,\frac{1}{4}d(z,w))$, if we apply the interior estimates \eqref{Interior Estimates} to $G_{k}u$, then for any $m\in \mathbb{N}$,
	\begin{equation*}
	\|\eta G_{k}u\|_{H^m(B(z,\frac{1}{k}))}
	\leq Ck^{m}\|G_{k}u\|_{L^2(B(z,\frac{1}{k}))}
	\leq Ck^me^{-b r f(k)\sqrt{k\log k}} \|u\|_{L^2(B(w,\frac{1}{4}d(z,w))},
	\end{equation*} 
	where $C$ is a constant only depending on $m$ and \k potential $\phi$. Taking $m=n+1$ and applying Sobolev embedding theorem,
	\begin{equation*}
	|G_{k}u(z)|_{h^{k}(z)}
	\leq 
	Ck^{n+1}e^{-b r f(k)\sqrt{k\log k}} \|u\|_{L^2(B(w,\frac{1}{4}d(z,w)))}.
	\end{equation*}  
	Therefore,
	\begin{equation*}
	\|G_{k}(z,\cdot)\|_{L^2(B(w,\frac{1}{4}d(z,w)))}
	=\sup_{u\in L^2(B(w,\frac{1}{4}d(z,w)))} \frac{\left|\int u(\zeta)G_{k}(z,\zeta)e^{-k\phi(\zeta)}d\zeta\right|_{h^{k}(z)}}{\|u\|_{L^2(B(w,\frac{1}{4}d(z,w)))}}\leq Ck^{n+1}e^{-b r f(k)\sqrt{k\log k}}.
	\end{equation*}
	Since for any $\zeta\in B(w,\frac{1}{4}d(z,w))$, we have
	\begin{equation*}
	\Box_{k}G_{k}(z,\zeta)=0,
	\end{equation*}
	where $\Box_{k}$ acts on the second component, again using the interior estimates \eqref{Interior Estimates} and Sobolev embedding theorem, we have
	\begin{equation*}
	|G_{k}(z,w)|_{h^{k}}
	\leq C k^{n+1}	\|G_{k}(z,\cdot)\|_{L^2(B(w,\frac{1}{4}d(z,w)))}
	\leq 
	Ck^{2n+2}e^{-b r f(k)\sqrt{k\log k}}.
	\end{equation*}
	Therefore,
	\begin{equation*}
	|G_{k}(z,w)|_{h^{k}}
	\leq
	Ck^{2n+2}e^{-b f(k)\sqrt{k\log k}\, d(z,w)}
	\leq 
	C'e^{-\tfrac{b}{2} f(k)\sqrt{k\log k}\, d(z,w)}.
	\end{equation*}	
	Note the constant $C$ only depends on the dimension $n$, constant $A$ as in \eqref{eq majorant} for $\partial^2\phi$ and positivity of $\sqrt{-1}\partial\bar{\partial}\varphi$. Then using a standard bootstrapping argument, we obtain the estimates for the $C^2$ norm of $G_k(z,w)$ with respect to $z$ and $w$ variables.  
\end{proof}

\subsection{Proof of Corollary \ref{Analytic}} In Theorem \ref{Estimate of G in a shrinking neighborhood}, we proved that in the particular case when $h$ is analytic we have for any $k\geq \kappa$
\begin{equation*}
	|\mathcal{G}_{k}(z,w)|+|\nabla_{z}\mathcal{G}_{k}(z,w)| + |\nabla_{z} \nabla_{w}\mathcal{G}_{k}(z,w)|\leq C e^{-bk |z-w|^2},
	\end{equation*} 
	whenever $\gamma\sqrt{\frac{\log k}{k}} \leq |z-w|\leq 1$, which in terms of the distance function on $X$ can be considered as $ \gamma \sqrt{\frac{\log k}{k}} \leq d(z, w) \leq \delta$ for some $\delta >0$ and a new $\gamma$ comparable with the previous one.   We then use the same iterative argument as above with $f(k)= \frac{k^\frac12}{\sqrt{\log k}}$, combined with a covering argument, to obtain
	$$|G_{k}(z,w)|_{C^2} \leq C e^{-b k \,d(z,w)} \quad \text{for} \quad d(z, w) \geq \delta, $$ which implies the same estimates for the Bergman kernel. But note that since $X$ is compact, in this estimate we can replace $d(z, w)$ with $d(z, w)^2$ by making $b$ a bit smaller; for example changing $b$ to $\frac{b}{\text{diam}(X)}$, with $\text{diam}(X)$ being the diameter of $X$, would do the job.

\section*{Acknowledgements} We are grateful to M. Christ for communicating to us a sketch of the proof of Corollary \ref{Analytic}. The second author would like to thank Z. Lu and B. Shiffman for their constant support and
mentoring.


\end{document}